\documentclass[11pt,reqno]{amsart}
\usepackage{fullpage}
\usepackage{times}
\usepackage[utf8]{inputenc}
\usepackage[english]{babel}
\usepackage[mathscr]{euscript}
\usepackage{amsmath,amssymb,amsthm, graphicx, enumitem, xcolor, verbatim}
\usepackage[colorlinks=true, linkcolor=blue, citecolor=blue, urlcolor=blue]{hyperref}
\usepackage{bbm}
\usepackage[normalem]{ulem}

\def\F{\mathbb{F}}

\def\R{\mathbb{R}}
\def\C{\mathbb{C}}

\def\N{\mathbb{N}}
\def\P{\mathbb{P}}
\def\E{\mathcal{E}}
\def\M{M}

\def\sgn{\mathop{\rm sgn}}
\def\intr{\mathop{\rm int}}

\newtheoremstyle{case}{}{}{}{}{}{:}{ }{}
\theoremstyle{case}

\theoremstyle{plain}

\newtheorem{theorem}{Theorem}[section]
\newtheorem{utheorem}{\textrm{\textbf{Theorem}}}

\newtheorem{corollary}[theorem]{Corollary}

\newtheorem{proposition}[theorem]{Proposition}
\newtheorem{lemma}[theorem]{Lemma}

\theoremstyle{definition}

\newtheorem{definition}[theorem]{Definition}
\newtheorem{remark}[theorem]{Remark}

\numberwithin{equation}{section}

\begin{document}
\title[Entrywise transforms preserving matrix positivity and non-positivity]{Entrywise transforms preserving matrix positivity and non-positivity}

\author[D.~Guillot]{Dominique Guillot}
\address[D.~Guillot]{University of Delaware, Newark, DE, USA and Universit\'e Laval, Qu\'ebec, QC, Canada}
\email{\tt dguillot@udel.edu}

\author[H.~Gupta]{Himanshu Gupta}
\address[H.~Gupta]{University of Regina, Regina, SK, Canada}
\email{\tt himanshu.gupta@uregina.ca}

\author[P.K.~Vishwakarma]{Prateek Kumar Vishwakarma}
\address[P.K.~Vishwakarma]{Universit\'e Laval, Qu\'ebec, QC, Canada}
\email{\tt prateek-kumar.vishwakarma.1@ulaval.ca,~prateekv@alum.iisc.ac.in}

\author[C.H.~Yip]{Chi Hoi Yip}
\address[C.H.~Yip]{Georgia Institute of Technology, Atlanta, GA, USA}
\email{cyip30@gatech.edu}


\begin{abstract}
We characterize real and complex functions which, when applied entrywise to square matrices, yield a positive definite matrix if and only if the original matrix is positive definite. We refer to these transformations as {\it sign preservers}. Compared to classical work on entrywise preservers of Schoenberg and others, we completely resolve this problem in the harder fixed dimensional setting, extending a similar recent classification of sign preservers obtained for matrices over finite fields. When the matrix dimension is fixed and at least \(3\), we show that the sign preservers are precisely the positive scalar multiples of the continuous automorphisms of the underlying field. This is in contrast to the $2 \times 2$ case where the sign preservers are extensions of power functions. These results are built on our classification of \(2 \times 2\) entrywise positivity preservers over broader complex domains. Our results yield a complementary connection with a recent work of Belton, Guillot, Khare, and Putinar on negativity-preserving transforms. We also extend our sign preserver results to matrices with a structure of zeros, as studied by Guillot, Khare, and Rajaratnam for the entrywise positivity preserver problem. Finally, in the spirit of sign preservers, we address a natural extension to monotone maps, classically studied by Loewner and many others.
\end{abstract}

\maketitle

\section{Introduction}
As we commemorate the 100th anniversary of the seminal observation by P\'olya and Szeg\H{o} in 1925, that every absolutely monotonic function preserves positive semidefiniteness when applied entrywise to real matrices, we address a different facet of the century-old classification problem in fixed dimensions, rooted in celebrated theorems of Schoenberg \cite{Schoenberg-Duke42}, Rudin \cite{Rudin-Duke59}, Herz \cite{herz1963fonctions}, and others. This story begins in 1911 with the fundamental Schur product theorem \cite{Schur1911}: for an integer $n\geq 1$, and positive semidefinite $A=(a_{ij})$ and $B=(b_{ij})\in \C^{n\times n}$, the entrywise product $A\circ B:=(a_{ij}b_{ij})$ is positive semidefinite. The class of positive semidefinite matrices forms a closed convex cone, which is further closed under taking entrywise conjugates. Thus if $A=(a_{ij})$ is positive semidefinite, then so is the entrywise product $A^{
\circ m}\circ \overline{A}^{\circ k}:=(a_{ij}^m\overline{a_{ij}}^{k})$ for integers $m,k\geq 0$, where we set $0^0:=1$. Writing $f[A] := (f(a_{ij}))$ for the entrywise action of a function $f: \C \to \C$ on a matrix $A = (a_{ij})$, we immediately conclude that the monomial $f(z) = z^m \overline{z}^k$ preserves positive semidefiniteness when applied entrywise. Note that this holds independently of the dimension of the matrix $A$. Clearly, the same holds for nonnegative combinations of such monomials, as well as appropriate limits of these polynomials. This is the observation of P\'olya and Szeg\H{o}, who asked in their famous book \cite{polyaszego} if other functions have the same property.

\begin{lemma}[P\'olya and Szeg\H{o} \cite{polyaszego}, Rudin \cite{Rudin-Duke59}]
Let $\Omega\subseteq \C$, and let $f: \Omega \to \C$ be given by $f(z)\equiv \sum_{m,k\geq 0}c_{m,k}z^{m}\overline{z}^{k},$ where $c_{m,k}\in [0,\infty)$ for all integers $m,k\geq 0$. Then $f[A]:=(f(a_{ij}))$ is positive semidefinite for all positive semidefinite $A\in \Omega^{n\times n},$ and all $n\geq 1.$
\end{lemma}

In the special case where $\Omega\subseteq \R,$ the corresponding set of functions identified above is the set of convergent power series with nonnegative coefficients, also known as absolutely monotonic functions. P\'olya and Szeg\H{o}'s work naturally led to the search for real functions that are not absolutely monotonic, but preserve positivity for all positive semidefinite matrices of all sizes. The problem was resolved 17 years later, through Schoenberg’s work and later refined by many, revealing — quite surprisingly — that no other such positivity preservers exist. To state the result formally, we introduce the following notation. For $\Omega \subseteq \C$, let $\P_n(\Omega)$ denote the set of all $n\times n$ Hermitian positive semidefinite matrices with entries in $\Omega$, and let $\P_n := \P_n(\C)$. We also define $D(a, r) := \{z \in \C : |z-a| < \rho\}$, the complex open disk centered at $a$ and of radius $\rho$. 

\begin{theorem}[Schoenberg \cite{Schoenberg-Duke42}, Rudin \cite{Rudin-Duke59}, Herz \cite{herz1963fonctions}, Vasudeva \cite{vasudeva1979positive}, {\cite[Chapter 18]{khare2022matrix}}]\label{Tschoenberg}
Let $0<\rho\leq \infty$. and let $\Omega=(-\rho,\rho),[0,\rho),(0,\rho)$, or $D(0,\rho)$. Let $f:\Omega\to\F,$ where $\F=\C$ if $\Omega\not\subseteq \R$ and $\F=\R$ otherwise. Then the following are equivalent:
\begin{enumerate}
    \item $f[A]:=(f(a_{ij}))\in \P_n$ for all $A\in \P_n(\Omega),$ and all $n\geq 1.$
    \item $f(z)\equiv \sum_{m,k\geq 0}c_{m,k}z^{m}\overline{z}^{k},$ where $c_{m,k}\geq 0$ for all integers $m,k\geq 0.$
\end{enumerate}
\end{theorem}
\noindent This result has sparked numerous classical and modern developments, with applications in areas as diverse as metric geometry and high-dimensional covariance estimation. See \cite{BGKP-survey-part-1, BGKP-survey-part-2} for details.

The \textit{dimension-free} aspect, i.e., the independence from the dimension $n$ of the preservers in Theorem~\ref{Tschoenberg} naturally leads to the problem of classifying positivity preservers $f$ over positive semidefinite matrices of a fixed dimension $n$. Other than the $n=2$ case resolved by Vasudeva \cite{vasudeva1979positive}, this \textit{fixed-dimensional} problem --- even after 100 years since P\'olya and Szeg\H{o}'s observation --- largely remains open. We refer to recent works of Khare and Tao \cite{Khare-Tao} and of Belton, Guillot, Khare, and Putinar \cite{BGKP-fixeddim} for developments involving symmetric function theory. In the same vein of fixed-dimensional results, an impactful classical result of FitzGerald and Horn is worth mentioning.

\begin{theorem}[FitzGerald and Horn \cite{fitzgerald1977fractional}; see also \cite{guillot2015complete,jain2020hadamard}]\label{thm:fitz_horn_fractional}
Fix an integer $n\geq 2$ and $\beta\in \R$. Then 
\begin{align*}
 A^{\circ \beta}:=(a_{ij}^{\beta})\in \P_n \quad\mbox{for all}\quad A=(a_{ij})\in \P_n([0,\infty))
\end{align*}
if and only if $\beta \in \N \cup [n-2,\infty)$. Moreover, if $\beta \in (0,n-2)\setminus \mathbb{N}$, then $ \bigl((1+\epsilon ij)^\beta\bigr)_{i,j=1}^n$ is not positive semidefinite for all sufficiently small $\epsilon>0$.
\end{theorem}

In a different direction, following the work of Cooper, Hanna, and Whitlatch \cite{cooper2022positive}, Guillot, Gupta, Vishwakarma, and Yip \cite{guillot2024positivity} recently considered positive definite matrices over finite fields $\F_q$. By analogy from $\R$, we call an element in $\F_q$ positive if it is the square of a non-zero element in $\F_q$; and we call an $n \times n$ symmetric matrix with entries in $\F_q$ positive definite if all its leading principal minors are positive in $\F_q$. Guillot, Gupta, Vishwakarma, and Yip \cite{guillot2024positivity} classified the entrywise positivity preservers for matrices over finite fields in the harder fixed-dimensional setting, specifically for matrices of a fixed dimension $n \geq 3$\footnote{In the $n=2$ case, the preservers are more field dependent and some cases of the problem remain open.}. Surprisingly, these are precisely the positive multiples of the field's automorphisms (i.e.,  powers of the Frobenius). To formalize this, henceforth, we denote the set of $n \times n$ matrices with entries in a set $S$ by $M_n(S)$.

\begin{theorem}[{\cite[Corollary 1.4]{guillot2024positivity}}]\label{TFiniteFields}
For a finite field $\F_q$ and a fixed $n \geq 3$, the entrywise positivity preservers on $M_n(\F_q)$ are precisely the positive multiples of the automorphisms of $\F_q$.
\end{theorem}

To the authors' knowledge, this is the only setting in which the fixed-dimensional entrywise positivity preserver problem has been fully resolved. Remarkably, in this case, the entrywise positivity preservers exhibit an even stronger property: a symmetric matrix $A$ is positive definite if and only if $f[A]$ is positive definite. We call such functions {\it sign preservers}. The characterization in Theorem~\ref{TFiniteFields} reveals that, over $\F_q$, the notions of positivity preservers and sign preservers are equivalent when $n\geq 3$.

\begin{corollary}\label{CsignFiniteFields}
Under the premise in Theorem~\ref{TFiniteFields}, the sign preservers on $M_n(\F_q)$ coincide with the positive multiples of the field automorphisms of $\F_q$. 
\end{corollary}

However, the structure of sign preservers becomes more intricate when restricted to the smaller class of
$2\times 2$ matrices over finite fields:

\begin{theorem}[{\cite[Theorem D]{guillot2024positivity}}]
For a finite field $\F_q$, the sign preservers on $\M_2(\F_q)$ are exactly: 
\begin{enumerate}
\item the bijective monomials, when $q$ is even.
\item the positive multiples of the field automorphisms of $\F_q$, when $q$ is odd.
\end{enumerate}
\end{theorem}

While entrywise positivity preservers have been intensively studied for matrices with real or complex entries \cite{BGKP-survey-part-1, BGKP-survey-part-2, khare2022matrix}, a systematic study of sign preservers was not previously carried out for real or complex fields. One of our goals in this paper is to fill this gap.   

\begin{definition}[Sign preservers]\label{defn:sign-preservers}
Let $\Omega \subseteq \C$. We say that a function $f: \Omega \to \C$ is a positive semidefinite {\it sign preserver} on $M_n(\Omega)$ if a Hermitian matrix $A \in M_n(\Omega)$ is positive semidefinite if and only if $f[A]$ is positive semidefinite. We similarly define positive definite sign preservers.
\end{definition}

Studying sign preservers is very natural. For example, as evidenced in \cite{belton2023matrix, belton2026negativity, guillot2024positivity}, generic entrywise preservers of positive semidefiniteness generally preserve supplementary structure such as strict positivity. As the cone of positive semidefinite matrices naturally decomposes into its positive and non-positive parts, it makes sense to examine how the entrywise action of a function can preserve each part of the cone separately.

As a consequence of our main results, we obtain the following analog of Corollary \ref{CsignFiniteFields} for the real and complex fields. Surprisingly, field automorphisms also appear naturally in this setting. Given how elusive solving the entrywise positivity preserver problem in fixed dimension has been as well as the general lack of results in the area in fixed dimension, it is interesting that the supplementary condition imposed by sign preservers is enough to obtain a full characterization.

\begin{theorem}\label{thm:main}
Fix $n\geq 3$, and let $\F = \R$ or $\C$. If $f: \F \to \F$, then the following are equivalent: 
\begin{enumerate}
\item $f$ is a positive definite sign preserver on $M_n(\F)$.
\item $f$ is a positive semidefinite sign preserver on $M_n(\F)$. 
\item $f$ is a positive multiple of a continuous field automorphism \footnote{It is well-known that continuous field automorphisms of $\C$ are precisely the identity map and the complex conjugation map. On the other hand, there do exist wild automorphisms of $\C$ that are far from continuous. We refer to Yale's paper~\cite{Y66} for details.} of $\F$, i.e., $f(x) = cx$ for some $c > 0$ when $\F = \R$, and $f(z) = cz$ or $f(z) = c \overline{z}$ for some $c > 0$ when $\F = \C$. 
\end{enumerate}
\end{theorem}
Note that a function $f$ preserves positive definiteness when applied entrywise to positive definite matrices if and only if it preserves the (positive) sign of the leading principal minors of the matrices. An interesting consequence of Theorem \ref{thm:main} is that, for matrices of dimension $3$ or above, a function $f$ is a sign preserver (as defined in Definition \ref{defn:sign-preservers}) if and only if it preserves the sign (negative, positive, or $0$) of all leading principal minors. For real matrices, sign preservers also coincide with the functions that preserve the sign of all square minors. 

As noted, a complete characterization of entrywise positivity preservers on \( n \times n \) matrices remains unknown for any fixed \( n \geq 3 \); strikingly, no conjectures have even been proposed to date, though notable partial results are available in \cite{BGKP-fixeddim,fitzgerald1977fractional,Khare-Tao}. Theorem~\ref{thm:main} has the following immediate notable consequence.

\begin{corollary}
Let $n\geq 3$, and let $\F=\R$ or $\C$. Any entrywise positivity preserver on $M_n(\F)$ that is not a positive scalar multiple of the most basic positivity-preserving transforms
\[
A \mapsto A \quad \text{and} \quad A \mapsto \overline{A}=A^T.
\]
must necessarily send at least one matrix that is not positive semidefinite to one that is.
\end{corollary}

Theorem \ref{thm:main} also connects naturally with recent work of Belton, Guillot, Khare, and Putinar \cite{belton2026negativity}, which classifies entrywise transforms in the \emph{dimension-free} setting sending Hermitian matrices with \emph{at most} \( k \) negative eigenvalues to those with \emph{at most} \( \ell \). In the case \( k = \ell = 1 \), they show that the only such preservers are affine maps of the form
\[
z \mapsto a + b z + c \overline{z}
\]
for \( a, b, c \in \mathbb{R} \) such that $\mathbbm{1}_{a < 0} + \mathbbm{1}_{b > 0} + \mathbbm{1}_{c > 0} \leq 1$, and where $\mathbbm{1}_A$ is the indicator function for condition $A$. Theorem~\ref{thm:main} offers a complementary result in the \emph{fixed-dimensional} setting for any fixed \( n \geq 3 \): an entrywise transform sends Hermitian matrices with \emph{at least} \( k = 1 \) negative eigenvalues to matrices with \emph{at least} \( \ell = 1 \) negative eigenvalues, and maps those with \emph{exactly} \( k = 0 \) negative eigenvalues to matrices with \emph{exactly} \( \ell = 0 \) negative eigenvalues, if and only if it is one of
\[
z \mapsto c z \quad \text{and} \quad z \mapsto c \overline{z}
\]
for some positive real constant \( c \). 

A natural variant of the positivity and sign preserver problems is to consider the entrywise action of functions on positive (semi)definite matrices with a given structure of zeros. Given an integer $n\geq 1$, let $G = (V,E)$ be a simple undirected graph on the vertex set $V = \{1, \dots, n\}$ with edge set $E \subseteq V \times V$. We define 
\[
M_G(\Omega):=\{A=(a_{ij})\in M_n(\C): \mbox{$a_{ij}=0$ for all $i \ne j$ such that  $(i,j) \not\in E$, and $a_{ij} \in \Omega$ otherwise}\}. 
\]

Hence, an entry $a_{ij}$ of $A \in M_G(\Omega)$ is constrained to be zero if $(i,j) \not \in E$ and $i \ne j$. The other entries are not constrained and, in particular, are allowed to be $0$ if $0 \in \Omega$. 
For a function $f:\Omega\to \C$, define $f_{G}[A]\in \M_n(\C)$ by
\[
(f_{G}[A])_{ij} := 
\begin{cases}
f(a_{ij}) & \mbox{if either $i=j$ or $(i,j) \in E$,}\\
0 & \mbox{otherwise.}
\end{cases}
\]

Functions that preserve positive semidefinitness when restricted to matrices in $M_G$ were previously studied by Guillot, Khare, and Rajaratnam \cite{GKR-critG, GKR-sparse, guillot2015functions}. In particular, in \cite{GKR-sparse}, they obtained a full characterization of the positivity preservers on $M_G([0,R))$ for an arbitrary tree $G$. Power functions that preserve positivity on $M_G$ were also studied in \cite{GKR-critG}, where the problem was fully solved for chordal graphs, as well as for some families of non-chordal graphs such as cycles and bipartite graphs. 

\begin{definition}[Sign preservers associated with graphs]
Given $\Omega \subseteq \C$, a function $f: \Omega \to \C$, and a simple undirected graph $G$ on $\{1,\dots,n\}$, we say $f_G$ is a positive semidefinite sign preserver on $M_G(\Omega)$ if a Hermitian matrix $A \in M_G(\Omega)$ is positive semidefinite if and only if $f_G[A]$ is positive semidefinite. We define positive definite sign preservers on $M_G(\Omega)$ in an analogous way. 
\end{definition}

In particular, observe that the special case $G = K_n$ (the complete graph on $n$ vertices) corresponds to the problems studied above. When $G$ has several connected components, say $G_1, \dots, G_k$,  observe that every matrix in $M_G(\Omega)$ can be reorganized in block diagonal form with diagonal blocks in $M_{G_i}(\Omega)$. It follows immediately that $f_G$ is a sign preserver on $M_G(\Omega)$ if and only if it is a sign preserver on $M_{G_i}(\Omega)$ for all $i=1,\dots,k$. Hence, it suffices to characterize sign preservers for connected graphs. We assume, without loss of generality, that graphs are connected throughout the paper.

\subsection{Main results}

We now state our main results, which provide a full characterization of sign preservers on $M_n(\Omega)$ for any fixed $n \geq 2$, for real and complex domains $\Omega$ satisfying mild conditions, as well as a full characterization of sign preservers on $M_G(\Omega)$ for any graph $G$ when $\Omega = \R, \C$, or $(0,\infty)$. In particular, Theorem~\ref{thm:main} follows immediately from Theorems~\ref{main_thm_2} and~\ref{main_thm_3} below. We begin with a technical definition.

\begin{definition}\label{def_pd_psd_type}
We call $\Omega\subseteq \C$ {\it reflection-symmetric} if $z \in \Omega$ implies $\overline{z} \in \Omega$, and {\it modulus-closed} if $z \in \Omega $ implies $|z| \in \Omega$. Moreover, if $I:=\Omega \cap [0,\infty)$, then we say that $\Omega$ is of
\begin{itemize}
\item {\it pd-type} if for all $z \in  \Omega \setminus I$, there exists $\epsilon>0$ such that $|z|+\epsilon \in I$; 
\item {\it psd-type} if for all $z \in  \Omega \setminus I$, there exists $\delta>0$ such that $|z|-\delta \in I$. 
\end{itemize}
\end{definition}

It may initially seem that the conditions imposed on \(\Omega\) in Definition~\ref{def_pd_psd_type}, and in our main results below, are somewhat peculiar. However, in light of the question under consideration, these assumptions are both necessary and mild. They are mild in the sense that the requirements on \(\Omega\) are quite flexible: for instance, \(\Omega\) can be any open real interval, an open disc, an open annulus, a star-shaped region, or a (finite) union of half-lines, provided the $\Omega$ is reflection symmetric, moduli-closed, and satisfies the other mild condition of pd- or psd-type. These examples illustrate the broad class of domains to which our results apply. On the other hand, these conditions are also necessary. For example, if \(\Omega \subseteq \mathbb{R}\) is of pd- or psd-type and contains both negative and positive elements, then it must satisfy \(|\inf \Omega| \leq \sup \Omega\). This inequality arises naturally when analyzing positivity preservers on \(M_2(\Omega)\), since for any \(x < -\sup \Omega\), no matrix \(A \in M_2(\Omega)\) that is positive semidefinite can include \(x\) as an entry. Consequently, the values of a function \(f\) on the interval \((-\infty, -\sup \Omega)\) are irrelevant for determining whether \(f\) preserves positivity on \(M_2(\Omega)\).

Our first main result fully characterizes sign preservers on $M_2(\Omega)$. 

\begin{utheorem}\label{main_thm_1}
Let $\Omega \subseteq \C$ be reflection symmetric, modulus-closed, of pd-type, and such that $I:=\Omega \cap [0,\infty)$ is an interval. Then the following are equivalent for a function $f:\Omega\to \C$:
\begin{enumerate}
    \item A Hermitian matrix $A\in M_2(\Omega)$ is positive definite if and only if $f[A]$ is positive definite. 
    \item There exist real $\alpha,\beta>0$ such that the following hold:
    \begin{itemize}
        \item For all $x\in \Omega\cap \R$, $f(x) =  \alpha\sgn(x) |x|^{\beta}$.
        \item For all $z \in \Omega\setminus\R$, we have
        \begin{align*}
         |f(z)| = \alpha|z|^\beta \quad \mbox{and} \quad f(\overline{z})= \overline{f(z)}.   
        \end{align*} 
    \end{itemize}
\end{enumerate}
The result holds verbatim with ``pd-type'' replaced by ``psd-type'' and ``definite'' by ``semidefinite''.
\end{utheorem}

Notice that the sign preservers over $\M_2(\R)$ are coincidentally the positive multiples of the \textit{odd extensions} of power functions, defined for $\beta > 0$ by 
\begin{align*}
    \psi_{\beta}(x):=\sgn(x)|x|^{\beta} \quad \mbox{for all } x\in \R\setminus\{0\},
\end{align*}
where $\psi_{\beta}(0):=0$. These extensions naturally arise in the development of the theory of \textit{critical exponents} of graphs, a framework connecting combinatorics and analysis, first studied by FitzGerald and Horn \cite{fitzgerald1977fractional}, and later expanded upon by Guillot, Khare, and Rajaratnam \cite{guillot2015complete, GKR-critG}. 

Our next main result addresses sign preservers on real domains for matrices of a fixed dimension $n \geq 3$.

\begin{utheorem}\label{main_thm_2}
Fix an integer $n\geq 3$. Let $\Omega \subseteq \R$ be modulus-closed, of pd-type, and such that $I:=\Omega \cap [0,\infty)$ is an interval. The following are equivalent for a function $f:\Omega\to \C$:
\begin{enumerate}
\item A symmetric matrix $A\in M_n(\Omega)$ is positive definite if and only if $f[A]$ is positive definite.
    \item There exists real $\alpha>0$ such that $f(x) =  \alpha x$ for all $x \in \Omega$.
\end{enumerate}
The result holds verbatim with ``pd-type'' replaced by ``psd-type'' and ``definite'' by ``semidefinite''.
\end{utheorem}

Our next main result characterizes the sign preservers for annuli in the complex plane. Here, by an annulus, we mean a set of the form $\{z \in \C:|z|\in I\}$ for an interval $I\subseteq [0,\infty)$. We note that our proof technique works for more general domains. However, for simplicity, we only state the result for annuli, and refer to Remark~\ref{remarkC} for more details.    

\begin{utheorem}\label{main_thm_3}
Fix an integer $n\geq 3$. Let $\Omega\subseteq \C$ be an annulus with $I:=\Omega \cap [0,\infty)$ an interval, and let $f:\Omega\to \C$. If $\Omega$ is of pd-type, then the following are equivalent:
\begin{enumerate}
\item A Hermitian matrix $A\in M_n(\Omega)$ is positive definite if and only if $f[A]$ is positive definite.
        \item There exists real $\alpha>0$ such that $f(z) \equiv \alpha z$ or $f(z) \equiv \alpha \overline{z}$ on $\Omega$.
\end{enumerate}
The result holds verbatim with ``pd-type'' replaced by ``psd-type'' and ``definite'' by ``semidefinite''.
\end{utheorem}

Our last main result characterizes sign preservers for matrices in $M_G(\Omega)$ for all connected graphs $G$. Surprisingly, we show that the characterization only depends on whether $G$ is a tree or not (or equivalently, whether $G$ is acyclic or $G$ contains a cycle). While our proof techniques can be adapted to more general domains, for simplicity, we only consider the cases where $\Omega$ is $\R$, $\C$, or $(0,\infty)$.  

\begin{utheorem}\label{TsignG}
Fix an integer $n \geq 3$, and suppose $G$ is a connected simple graph on $n$ vertices. Let $f:\Omega\to \C$, where $\Omega$ is any of $\R$, $\C$, or $(0,\infty)$. Then
\begin{enumerate}
\item If $G$ is a tree, then the following are equivalent:
\begin{enumerate}
\item A Hermitian matrix $A\in M_G(\Omega)$ is positive definite if and only if $f_G[A]$ is positive definite.
    \item $f$ is precisely among the preservers in Theorem~\ref{main_thm_1} with $\alpha>0$ and $\beta = 1$.
\end{enumerate}
\item If $G$ is not a tree, then the following are equivalent: 
\begin{enumerate}
\item A Hermitian matrix $A\in M_G(\Omega)$ is positive definite if and only if $f_G[A]$ is positive definite.

    \item Either $f(z) \equiv \alpha z$ or $f(z) \equiv \alpha \overline{z}$ for some $\alpha > 0$. 
\end{enumerate}
\end{enumerate}
The same holds if ``positive definite'' is replaced by ``positive semidefinite''.
\end{utheorem}

We also append our main results with a natural variant of these sign preservers, addressed in Theorem~\ref{main_thm_1-oldcase} and Remark~\ref{rem:M_n-case}.

Finally, a natural extension of the positivity preserver problem involves studying functions that are monotone with respect to the Loewner ordering. Given two $n \times n$ positive semidefinite matrices, we write $A \geq B$ whenever $A-B$ is positive semidefinite, and $A > B$ whenever $A-B$ is positive definite. Functions satisfying $f[A] \geq f[B]$ whenever $A \geq B$ are two positive semidefinite matrices were previously studied, for instance in \cite{fitzgerald1977fractional, guillot2015complete, hiai2009monotonicity}. In particular, it is well-known that a power function $f(x) = x^\alpha$ is monotone on $M_n([0,\infty))$ if and only if $\alpha \in \N \cup [n-1, \infty)$. In the spirit of Theorems \ref{main_thm_1}, \ref{main_thm_2}, and \ref{main_thm_3}, and as an application, in Section \ref{Smonotone} we characterize functions $f$ such that $f[A] \geq f[B]$ if and only if $A \geq B$, and likewise satisfying $f[A] > f[B]$ if and only if $A > B$.

\medskip

\textbf{Outline of the paper.} The rest of the paper is organized as follows. In Section \ref{Sprelim}, we review classical characterizations of $2\times 2$ positivity preservers and extend them to broader domains. In Section \ref{sec2}, we prove our characterization of sign preservers over $\M_2$ (Theorem \ref{main_thm_1}). Section \ref{Sreal3} contains the proof of our characterization of real sign preservers for matrices of dimension three or higher (Theorem \ref{main_thm_2}), while in Section \ref{Scomplex3} we prove the corresponding result for complex matrices (Theorem \ref{main_thm_3}). We prove Theorem~\ref{TsignG} in Section~\ref{sec:graphs}.  Finally, in Section \ref{Smonotone}, we discuss monotone transformations with respect to the Loewner ordering.

\section{Entrywise positivity preservers}\label{Sprelim}

We begin with a classification of positive (semi)definite preservers on $2\times 2$ matrices.

\subsection{Positivity preservers over $2\times 2$ matrices}

Let $I \subseteq [0,\infty)$ be an interval and $f: I \to [0,\infty)$. We call $f$ \emph{multiplicatively midconvex} if for all $x, y \in I$,  
\[
\sqrt{f(x)f(y)} \geq f(\sqrt{xy}).
\]   
For a real-valued function $f$, we denote the right-hand limit at $x$ by $f^{+}(x)$ whenever it exists.  

The following result has been established for certain special intervals \cite[Chapter 6]{khare2022matrix}. Here, we extend the result to more general domains.

\begin{proposition}\label{2by2preservers}
Let $\Omega \subseteq \C$ be a reflection-symmetric set that is modulus-closed such that $I:=\Omega\cap[0,\infty)$ is an interval. Let $f: \Omega \to \C$. Then the following are equivalent:
\begin{enumerate}
    \item $f[A]$ is positive definite for all $A\in \M_2(\Omega)$ positive definite.
    \item The following conditions hold:
    \begin{enumerate}
        \item $f|_{I\setminus\{0\}}$ is real-valued, positive, increasing, and multiplicatively midconvex;
        \item if $0\in I$ then the right hand limit over reals $f^{+}(0)\geq |f(0)|$; and
        \item for all $z \in \Omega\setminus\{0\}$ such that there exists $\epsilon>0$ with $|z|+\epsilon \in I$, we have
        \begin{align*}
         f(|z|) \geq |f(z)| \quad \mbox{and}\quad    f(\overline{z}) = \overline{f(z)}.
        \end{align*}
    \end{enumerate}
\end{enumerate}
In particular, $f|_{I\setminus\{0\}}$ is continuous except possibly at the upper-end point of $I$ (if it belongs to $I$).
\end{proposition}
\begin{proof}
The result is true vacuously if $I\setminus\{0\}=\emptyset$, so suppose it is not.\\ 
\noindent $(2)\implies (1)$: Pick any positive definite $A:=\begin{pmatrix}a & z \\ \overline{z} & c\end{pmatrix}\in M_2(\Omega)$. Notice $a, c > 0$. If $z = 0$, then $f(a) f(c) > f^+(0) \geq |f(0)|^2 = f(0) \overline{f(0)}$. Thus, $f[A]$ is positive definite.  Now, suppose $z \ne 0$. Since $a, c \in I$ and $I$ is an interval, we have $\sqrt{ac} \in I$. Moreover, $|z| < \sqrt{ac}$ and so there exists $\epsilon > 0$ such that $|z| + \epsilon \in I$. Hence, using (2c), we obtain  $f(z)f(\overline{z})=f(z)\overline{f(z)}=|f(z)|^2 \leq f(|z|)^2 < f(\sqrt{ac})^2 \leq f(a)f(c)$. It follows that $f[A]$ is positive definite. 
\medskip

\noindent $(1)\implies (2)$: To prove (a) and (b), consider the restriction of $f$ to $\Omega\cap\R$. For $0<x<y\in I$, the matrix $\begin{pmatrix}y & x \\ x & x\end{pmatrix}$ is positive definite. Therefore $\begin{pmatrix}f(y) & f(x) \\ f(x) & f(x)\end{pmatrix}$ is positive definite. Thus, $f$ is positive and increasing over $I \setminus \{0\}$. From this, $f$ has at most countably many discontinuities over $I$, and all these are jump discontinuities. Use $\intr(I)$ to denote the interior points of $I$ as a subset of $\R$ and define 
\[
f^+(x):= \lim_{y\to x^+} f(y).
\]
By monotonicity, we have $f^+(x)\geq f(x)$ for all $x\in \intr(I)$. We will show that $f^+$ is continuous on $\intr(I)$. Let $x, y \in \intr(I)$ with $x\neq y$. Then $\sqrt{xy}\in \intr(I)$ and for small $\epsilon>0$,
\begin{align*}
A(\epsilon):=\begin{pmatrix}x+\epsilon & \sqrt{xy} + \epsilon \\ \sqrt{xy} + \epsilon & y+ \epsilon \end{pmatrix}\in M_2(I)
\end{align*}
is positive definite. Therefore, $f[A(\epsilon)]$ is positive definite. Thus $f(x+\epsilon)f(y+ \epsilon) > f(\sqrt{xy} + \epsilon)^2.$ Taking $\epsilon \to 0^+$, we obtain $f^+(x)f^+(y) \geq f^+(\sqrt{xy})^2$ for all $x,y\in \intr(I).$ Thus $f^+(x)$ is multiplicatively midconvex over $\intr(I)$ (along with being positive and nondecreasing). Let $\log \intr(I):=\{\log p: p\in \intr(I)\}$ and define the function $g:\log \intr(I) \to \R$ via $g(x):=\log f^+(e^x)$ for all $x \in \log \intr(I)$. Then $g$ is midconvex and nondecreasing on $\log \intr(I).$ In particular, it is bounded on compact sets in $\log \intr(I)$. Therefore, by \cite{roberts1974convex} or \cite[Theorem~6.2]{khare2022matrix}, $g$ is continuous over $\log \intr(I)$. It follows that $f^+(x)$ is continuous over $\intr(I)$. Thus, $f$ is continuous and multiplicatively midconvex over $\intr(I)$.

Next, suppose $0 \in I$. Considering the matrix $\begin{pmatrix}x & 0 \\ 0 & x\end{pmatrix}$ with $0 < x \in I$, we obtain $f(x)>|f(0)|$. Since $f$ is increasing on $I\setminus \{0\}$, it follows that $f^{+}(0)\geq |f(0)|.$ We now examine the mid-convexity of $f$ at the end points of $I$. Denote by $\lambda,\upsilon$ the lower- and upper-end points of $I$ respectively. We first show that $f$ is right continuous at $\lambda\in I$ if $\lambda>0.$ Note that $f^+(\lambda)$ exists and $f^+(\lambda) \geq f(\lambda) > 0$ since $\lambda > 0$. The matrix 
\begin{align*}
A(\epsilon):=\begin{pmatrix}
\lambda +\epsilon & \lambda + \epsilon/3 \\
\lambda + \epsilon/3 & \lambda
\end{pmatrix} \in M_2(I)
\end{align*}
is positive definite for all small $\epsilon>0.$ Therefore $\det f[A(\epsilon)] = f(\lambda+\epsilon)f(\lambda) - f(\lambda+\epsilon/3)^2 >0.$ Letting ${\epsilon \to 0^+}$ yields $f^+(\lambda) f(\lambda) - {f^+(\lambda)}^2\geq 0.$ Since $f^+(\lambda)>0,$ we obtain $f^+(\lambda)  \leq f(\lambda)$. Thus $f^+(\lambda) = f(\lambda)$. Now, for $x \in I \setminus \{0, \upsilon\}$ and $y \in I \setminus \{0\}$, consider the positive definite matrix $B(x,y,\epsilon):=\begin{pmatrix}x+\epsilon & \sqrt{xy} \\ \sqrt{xy} & y \end{pmatrix}\in M_{2}(I)$ for small $\epsilon >0$. Using the right continuity at $x$, we conclude that $f$ is multiplicatively midconvex on $I \setminus \{0\}$. For the final step, suppose $z\in \Omega\setminus\{0\}$ is such that there exists $\epsilon>0$ with $|z|+\epsilon\in I$. Then 
\begin{align*}
B(\epsilon):=\begin{pmatrix}|z|+\epsilon & z \\ \overline{z} & |z|\end{pmatrix}\in M_2(I)    
\end{align*}
is positive definite. Therefore $f[B(\epsilon)]$ is positive definite, and in particular $f(\overline{z})=\overline{f(z)}.$ Moreover, $f(|z|+\epsilon)f(|z|) > |f(z)|^2$. Using that $f$ is increasing and continuous on $I \setminus \{0\}$, we conclude that $f(|z|) \geq |f(z)|$. This completes the proof. 
\end{proof}

Notice that if $f$ maps positive definite matrices in $M_2((0,\infty))$ into positive semidefinite matrices, then the proof of Proposition \ref{2by2preservers} shows that $f$ must be continuous. We record this fact as a corollary as it will be useful in Section \ref{sec:graphs}.
\begin{corollary}\label{Cpdtopsd}
Let $f: (0,\infty) \to \R$ and suppose $f[A]$ is positive semidefinite for all $A \in M_2((0,\infty))$ positive definite. Then $f$ is continuous on $(0,\infty)$. 
\end{corollary}

The next result is analogous to Proposition \ref{2by2preservers}, but for transformations that preserve positive semidefiniteness.

\begin{proposition}\label{2by2preservers:psd}
Let $\Omega \subseteq \C$ be a reflection-symmetric set that is modulus-closed such that $I:=\Omega\cap[0,\infty)$ is an interval. Let $f: \Omega \to \C$. Then the following are equivalent:
\begin{enumerate}
    \item $f[A]$ is positive semidefinite for all $A\in M_2(I)$ positive semidefinite.
    \item The following conditions hold:
    \begin{itemize}
        \item $f|_{I}$ is real-valued, nonnegative, nondecreasing, and multiplicatively midconvex; and
        \item for all $z\in \Omega$, we have $f(|z|) \geq |f(z)|$ and $f(\overline{z}) = \overline{f(z)}$.
    \end{itemize}
\end{enumerate}
Moreover, $f|_{I\setminus\{0\}}$ is either identically zero or never zero except possibly at the upper-end point of $I$ (if it belongs to $I$). In particular, $f|_{I\setminus\{0\}}$ is continuous except possibly at the upper-end point of $I$ (if it belongs to $I$). 
\end{proposition}
\begin{proof}
The result is true vacuously if $I=\emptyset$, so suppose it is not. One can see that $(1)$ implies $(2)$ using the matrices
\[
\begin{pmatrix}
x & 0 \\
0 & x
\end{pmatrix},\quad 
\begin{pmatrix}
y & x \\
x & y
\end{pmatrix},\quad
\begin{pmatrix}
x & \sqrt{xy} \\
\sqrt{xy} & y
\end{pmatrix}, \quad\mbox{and}\quad
\begin{pmatrix}
|z| & z \\
\overline{z} & |z| 
\end{pmatrix},
\]
where $x\leq  y \in I$, $z \in \Omega$. The proof of $(2)$ implies $(1)$ is similar to that of its counterpart in Proposition~\ref{2by2preservers} and we omit the details.

We show the `never zero or always zero' part. Suppose $f(x)=0$ for some positive $x\in I.$ As $f$ is nonnegative and nondecreasing, we have $f\equiv 0$ over $(0,x]\cap I.$ we claim that $f(y)=0$ for $x<y\in I.$ Following the idea in \cite[Theorem~6.7]{khare2022matrix}, choose some large integer $n$ such that $w:=\sqrt[n]{y/x}<\upsilon/y,$ where $\upsilon$ is the upper-end point of $I$ and $y<v$. Consider the following matrices:
\begin{align*}
A_k:=\begin{pmatrix}xw^{k-1} & x w^{k} \\ x w^{k} & xw^{k+1}\end{pmatrix}\quad \mbox{for }k=1,\dots,n.
\end{align*}
By construction each $A_{k}\in M_2(I)\subseteq \M_2(\Omega)$ is positive semidefinite. Since $\det f[A_{k}]\geq 0$, we obtain:
\begin{align*}
    0\leq f(xw^k)\leq \sqrt{f(xw^{k-1})f(xw^{k+1})}\quad \mbox{for $k=1,\dots,n.$}
\end{align*}
Since $f(x)=0,$ using the inequality for $k=1,$ we get $f(xw)=0.$ Following this inductively, we get $f(xw^{k})=0$ for all $k=1,\dots,n,$ and in particular $f(y)=f(xw^{n})=0,$ concluding this proof.

The proof of the continuity of $f|_{I \setminus \{0\}}$ is similar to the proof of the analogous result in Proposition~\ref{2by2preservers}.
\end{proof}

\subsection{Extension of positive definite/positive semidefinite matrices}

We isolate two positive definite extension problems which will be useful later. For a Hermitian square matrix $A$, and a real number $x$, consider the following extension of the matrix $A$: 
\begin{align}\label{eqn_extension_matrix}
\E(A;x)= \begin{pmatrix}
A & \mathbf{v}^* \\
\mathbf{v} & x
\end{pmatrix},
\end{align}
where $\mathbf{v}$ denotes the last row of $A$ and $\mathbf{v}^*=\overline{\mathbf{v}}^T$. Moreover, if $x$ is equal to the bottom right entry of $A$, we write $\E(A)=\E(A;x)$.

\begin{lemma}\label{LPDextension}
Fix an integer $n\geq 2$ and a positive interval $I\subseteq \R$ such that $\sup I \not\in I$. Let $\Omega\subseteq \C$ such that $\Omega \cap (0,\infty)=I$. If $A \in M_n(\Omega)$ is positive definite, then for any $m>n$, there exists $A' \in M_m(\Omega)$ positive definite such that its $n \times n$ leading principal submatrix is $A$.
\end{lemma}
\begin{proof}
Let $a_{nn}$ denote the $(n,n)$ entry of $A$ and let $x\in \Omega \cap \R$ with $x>a_{nn}$. Observe that 
\[
\det \E(A;x) = (x-a_{nn}) \cdot \det A>0. 
\]
Thus, $\E(A;x)$ is a positive definite matrix in $M_{n+1}(\Omega)$ whose $n \times n$ leading principal submatrix is $A$. Applying the same argument recursively yields the desired matrix $A'$. 
\end{proof}

The next lemma is analogous to Lemma \ref{LPDextension}, but for positive semidefinite matrices.

\begin{lemma}\label{LPSDextension}
Fix an integer $n\geq 2$ and let $\Omega \subseteq \C$. If $A \in M_n(\Omega)$ is positive semidefinite, then for any $m > n$, there exists $A' \in M_m(\Omega)$ positive semidefinite such that its $n \times n$ leading principal submatrix is $A$. 
\end{lemma}

\begin{proof}
By induction, it suffices to prove the case $m=n+1$. Let $B=\E(A)$; then $B \in M_{n+1}(\Omega)$ and its $n \times n$ leading principal submatrix is $A$. We claim that $B$ is positive semidefinite. Observe that the last two rows of $B$ are exactly the same. By Sylvester's criterion, a Hermitian matrix is positive semidefinite if and only if all principal minors are nonnegative. Let $I$ be a nonempty subset of $\{1,2,\ldots, n+1\}$; we need to show the principal minor $\det B_{I,I}\geq 0$. 
 We consider the following:
 \begin{enumerate}
     \item If $n+1\notin I$, then we know $B_{I,I}=A_{I,I}$ and thus $\det B_{I,I}\geq 0$. 
     \item If $n+1\in I$ and $n \notin I$, then we know $B_{I,I}=A_{I',I'}$ with $I'=(I \setminus \{n+1\}) \cup \{n\}$ and thus $\det B_{I,I}\geq 0$.
     \item If $n,n+1\in I$, then $\det B_{I,I}=0$ since the last two rows of $B_{I,I}$ are exactly the same.
 \end{enumerate}
This completes the proof.
\end{proof}

\section{The $2\times 2$ sign preservers}\label{sec2}

We prove Theorem \ref{main_thm_1}, beginning with the positive definite case. 

\begin{proof}[Proof of Theorem~\ref{main_thm_1} (the positive definite case)] $(\Longleftarrow)$ Let $A = \begin{pmatrix}a & z \\ \overline{z} & c\end{pmatrix} \in M_2(\Omega),$ with $a,c\in \R$.  Suppose $A$ is positive definite, i.e., $a > 0$, and $ac > |z|^2=z\overline{z}$. Then $\alpha a^{\beta}> 0,$ and so $f(a)f(c) = \alpha^2 a^{\beta}c^{\beta} > \alpha^2|z|^{2\beta} = |f(z)|^2 = f(z) f(\overline{z})$. This proves $f[A]$ is positive definite. On the other hand, suppose $A$ is not positive definite. If $a\leq 0$ or $c\leq 0$, then $f(a)\leq 0$ or $f(c)\leq 0$. If $a,c \in (0, \infty)$  and $ac \leq |z|^2=z\overline{z}$, then 
$f(a) f(c) = \alpha^2 a^\beta c^{\beta} \leq \alpha^2|z|^{2\beta}= |f(z)|^2 = f(z) f(\overline{z})$. In both cases, the matrix $f[A]$ is not positive definite.

\noindent $(\Longrightarrow)$ We begin by examining the properties of $f$ over $\R$. By Proposition~\ref{2by2preservers}, $f$ is positive, increasing, and multiplicatively midconvex over $I\setminus\{0\}.$ On the other hand, the matrix $A(x,y):=\begin{pmatrix}x & \sqrt{xy} \\ \sqrt{xy} & y\end{pmatrix}\in \M_2(\Omega)$ is not positive definite for all positive $x,y\in I.$ Therefore $f[A(x,y)]$ is not positive definite. Since $f(x),f(y)>0,$ we must have $f(x)f(y)\leq f(\sqrt{xy})^2$ for all $x,y\in I\setminus\{0\}.$ Thus 
\begin{align}\label{midconv-iden}
f(\sqrt{xy})^2=f(x)f(y)
\end{align}
for all $x,y\in I\setminus\{0\}$. Proposition~\ref{2by2preservers} also gives that $f$ is continuous over $I':=\intr(I)$, the interior of $I$. Equation~\eqref{midconv-iden} yields that for $\phi(x):=\log f(e^x)$ for all $x\in \log I':=\{\log x: x\in I'\}$, we have 
\[
\frac{\phi(x)+\phi(y)}{2}=\phi\left(\frac{x+y}{2}\right)
\] 
for all $x,y\in \log I'$. Moreover, $\phi$ is continuous over $I'$. Therefore $\phi$ is affine (see \cite[Chapter VII]{roberts1974convex}). Hence $f$ is of the required form over $I'.$ To further show that $f$ has the required form over all points in $I,$ we show that $(i)$ $f(0)=0$ if $0\in I,$ and $(ii)$ $f$ is left continuous at the upper-end point of $I$ provided it belongs to $I.$

For $(i)$: From Proposition~\ref{2by2preservers}, we have that the right hand limit $f^+(0)\geq |f(0)|$. Moreover, since $f$ is of the required form over $I'$, we have $f^+(0) = 0$, showing $f(0)=0$.


For $(ii)$: Suppose $\upsilon \in I$ is the upper-end point. Note that $B(\epsilon):=\begin{pmatrix} \upsilon & \upsilon-\epsilon \\ \upsilon-\epsilon & \upsilon-3\epsilon \end{pmatrix} \in M_2(I)$ is not positive definite for small $\epsilon>0$. As  $f(\upsilon)>0$, for all small $\epsilon>0$ we must have $\det f[B(\epsilon)] = f(\upsilon -3\epsilon)f(\upsilon) - f(\upsilon-\epsilon)^2 \leq 0.$ Since $\lim_{\epsilon \to 0^+}\det f[B(\epsilon)] \leq 0,$ we get the left hand limit $f^{-}(\upsilon)\geq f(\upsilon),$ and monotonicity implies the left continuity at $\upsilon$. This proves $f$ has the desired form on $I$.

We now examine the properties of $f$ over $\C\setminus I$. Proposition~\ref{2by2preservers} implies that $f(\overline{z})=\overline{f(z)}$ and $f(|z|) \geq |f(z)|$ for all $z\in \Omega$, since $\Omega$ is pd-type. In particular, $f(\Omega \cap \R) \subseteq \R$. 

In what follows (and later), we use the following test matrices for $a,z,w,u\in \Omega$ and $\epsilon\geq 0$:
\begin{align}\label{test-matrices-1}
A(z,w,\pm \epsilon):=\begin{pmatrix}|z| \pm \epsilon & z \\ w & |z|\end{pmatrix},\quad B(a,\epsilon):=\begin{pmatrix} a & |a| \\ |a| & |a| + \epsilon\end{pmatrix},\quad C(z,w,u,\epsilon):=
\begin{pmatrix}
|z|+\epsilon & z\\
w & u
\end{pmatrix}. 
\end{align}

We already have that $\alpha|z|^{\beta}=f(|z|) \geq |f(z)|$ for all $z\in \Omega$. We show that the reversed inequality also holds. Note that $A(z,\overline{z},0) \in M_2(\Omega)$, and is not positive definite. Therefore, $f[A(z,\overline{z},0)]$ is also not positive definite, and $\det f[A(z,\overline{z},0)]\leq 0$. This gives $f(|z|)\leq |f(z)|$. Hence, $\alpha|z|^{\beta}=f(|z|) = |f(z)|$ for all $z \in \Omega\setminus I$. Now, let $a \in \Omega \cap (-\infty, 0)$. Observe that $f(a) \ne 0$ since $|f(a)| = \alpha |a|^\beta$. Suppose $f(a) > 0$. Then the matrix $B(a,\epsilon)\in \M_2(\Omega)$ is not positive definite, but $f[B(a,\epsilon)]$ is, a contradiction. Therefore we must have $f(a) < 0$ and so $f(a) = -\alpha |a|^\beta = \alpha \sgn(a) |a|^\beta$, as claimed.
\end{proof}

Most of the next proof for the positive semidefinite case is similar to the previous positive definite case, so we only provide those arguments that are different.

\begin{proof}[Proof of Theorem~\ref{main_thm_1} (the positive semidefinite case)]$(\Longleftarrow)$ The proof is similar to the positive definite case.

\noindent $(\Longrightarrow)$ We begin by examining the properties of $f$ over $I$. Clearly $f\not\equiv 0$ over $I\setminus\{0\}$, thus by Proposition~\ref{2by2preservers:psd} $f(x)>0$ for all $x\in I\setminus\{0\}$. Moreover, by Proposition~\ref{2by2preservers:psd}, $f$ is nonnegative, nondecreasing, and multiplicatively midconvex over $I\setminus\{0\}$. Furthermore, $f|_{I\setminus\{0\}}$ is continuous over $I\setminus\{0\}$ (except possibly at its upper-end point, if it is in $I\setminus\{0\}$).

The matrix $D(x,y,\epsilon):=\begin{pmatrix}x & \sqrt{xy}+\epsilon \\ \sqrt{xy}+\epsilon & y\end{pmatrix}\in M_2(I)$ is not positive semidefinite for all $x,y\in I\setminus\{0\}$, and small $\epsilon>0.$ Therefore $f[D(x,y,\epsilon)]$ is not positive semidefinite. Since $f(x),f(y)>0$, using the continuity of $f|_{I \setminus \{0\}}$, we must have $f(x)f(y)\leq f(\sqrt{xy})^2$. Thus 
$
f(\sqrt{xy})^2=f(x)f(y)
$
for all $x,y\in I\setminus\{0\}$. Proceeding as in the proof of the positive definite case of Theorem~\ref{main_thm_1}, we obtain that $f$ is of the required form over interior points in $I$ (as a subset of $\R$). The same argument as in the proof of the positive definite case also shows $f$ has the required form at the upper-end point of $I$ (provided it belongs to $I$). Next we show that $f(0)=0$ if $0\in I.$ Applying $f$ to the zero matrix ${\bf 0}_{2 \times 2}$ shows that $f(0)\geq 0.$ If $f(0)>0,$ then consider $A:=\begin{pmatrix} 0 & a \\ a & 0 \end{pmatrix}\in M_{2}(I)$ for small $a>0$ such that $\alpha a^{\beta}\leq f(0).$ Then $A$ is not positive semidefinite but $f[A]$ is, a contradiction. Thus $f(0)=0.$ This proves $f$ has the claimed properties over $I$.

We now examine the properties of $f$ over $\Omega \setminus I$. Proposition~\ref{2by2preservers:psd} implies that $f(\overline{z})=\overline{f(z)}$ and $|f(z)|\leq f(|z|)=\alpha |z|^{\beta}$ for all $z\in \Omega\setminus I$. To show the remaining properties, we employ the test matrices in equation~\eqref{test-matrices-1}. We already have $|f(z)|\leq f(|z|)=\alpha |z|^{\beta}$ for all $z\in \Omega\setminus I$. To show the other inequality, consider $A(z,\overline{z},-\epsilon)\in \M_2(\Omega)$ which is not positive semidefinite for all small $\epsilon>0$. Therefore we must have $\det f[A(z,\overline{z},-\epsilon)]<0$, which using continuity of $f|_{I\setminus\{0\}}$, implies $f(|z|)\leq |f(z)|$. Therefore $|f(z)|=\alpha|z|^{\beta}$.

For $a\in \Omega\cap (-\infty,0)$, using $B(a,0)$ we show that $f(a)=-\alpha|a|^{\beta}$, as required.
\end{proof}

Theorem~\ref{main_thm_1} classifies functions that preserve both the positive (semi)definite and non-positive (semi)\-definite matrices \emph{within} the Hermitian part of $M_2(\Omega)$. A natural variant of this result considers functions that still preserve the positive (semi)definite matrices, but now allow preservation of the larger class of \textit{all} non-positive (semi)definite matrices in $M_2(\Omega)$, i.e. including those that are not Hermitian. In particular, we now allow a non-positive definite Hermitian matrix to map to a non-Hermitian one. While it is not immediately clear why the preservers in the former setting should contain those of the latter, our next result shows that this is indeed the case. More precisely, the preservers in the latter setting still belong to the same class described in Theorem~\ref{main_thm_1}, but they now satisfy additional injectivity conditions, resulting in a strict subclass of the previously identified preservers.

\begin{theorem}\label{main_thm_1-oldcase}
Under the premise of Theorem~\ref{main_thm_1}, the following are equivalent:
\begin{enumerate}
    \item $f[A]$ is positive definite if and only if $A\in M_2(\Omega)$ is positive definite.
    \item There exist real $\alpha,\beta>0$ such that $f$ is among the preservers in Theorem~\ref{main_thm_1}, and such that $f$ is injective and satisfies $f(\Omega\setminus \R)\subseteq \C\setminus \R$.
 \end{enumerate}
The result holds verbatim with ``pd-type'' replaced by ``psd-type'' and ``definite'' by ``semidefinite''.
\end{theorem}

\begin{proof}(The positive definite case) $(\Longleftarrow)$ Let $A = \begin{pmatrix}a & z \\ w & c\end{pmatrix}\in M_2(\Omega).$  Suppose $A$ is positive definite, i.e., $w=\overline{z}$, $a > 0$, and $ac > |z|^2=z\overline{z}$. Then $\overline{f(z)} = f(\overline{z})$, $\alpha a^{\beta}> 0,$ and so $f(a)f(c) = \alpha^2 a^{\beta}c^{\beta} > \alpha^2|z|^{2\beta} = |f(z)|^2 = f(z) f(\overline{z})$. This proves $f[A]$ is positive definite. On the other hand, suppose $A$ is not positive definite. If $w \ne \overline{z}$, then $f(w) \ne \overline{f(z)}$ since $\overline{f(z)} = f(\overline{z})$ and $f$ is injective. Thus, $f[A]$ is not positive definite. Now assume $w = \overline{z}$. If $a \not\in (0,\infty)$ or $c \not\in (0,\infty)$, then $f(a)\not\in (0,\infty)$ or $f(c) \not \in (0,\infty)$. If $a,c \in (0, \infty)$ and $ac \leq |z|^2=z\overline{z}$, then 
$f(a) f(c) = \alpha^2 a^\beta c^{\beta} \leq \alpha^2|z|^{2\beta}= |f(z)|^2 = f(z) f(\overline{z})$. In all cases, the matrix $f[A]$ is not positive definite.

$(\Longrightarrow)$ This is the proof of Theorem~\ref{main_thm_1} (positive definite case) appended by the following two paragraphs using the test matrices defined in equation~\eqref{test-matrices-1}. Consider $z\in \Omega\setminus\R$ and assume $f(z)\in \R$. Then $f(\overline{z})=\overline{f(z)} = f(z)\in \R$. If $f(z)>0$ then $f(z)=\alpha|z|^{\beta}$. Then $C(z,\overline{z},z,\epsilon)\in \M_2(\Omega)$ is not positive definite, but $f[C(z,\overline{z},z,\epsilon)]$ is, a contradiction.
On the other hand, if $f(z)< 0$ then $f(z)=-\alpha|z|^{\beta}$, in which case, consider $A(z,z,\epsilon)\in \M_2(\Omega)$. The matrix $A(z,z,\epsilon)$ is not positive definite, but $f[A(z,z,\epsilon)]$ is positive definite, a contradiction. Therefore, we must have $f(z)\in \C\setminus\R$.

Finally, we show that $f$ is injective on $\Omega\setminus \R$. Let $z, w \in \Omega \setminus \R$ with $z \ne w$, and assume $f(z) = f(w)$. If $z = \overline{w}$, then $f(z) = f(\overline{w}) = \overline{f(w)} \ne f(w)$ since $f(w) \not\in \R$. If instead $z \ne \overline{w}$, then consider the matrix $A:=A(z,\overline{w},\epsilon)\in \M_2(\Omega)$.
Clearly, $A$ is not positive definite. However 
\[
f[A]=
\begin{pmatrix}
f(|z|+\epsilon) & f(z)\\
f(\overline{w}) & f(|z|)
\end{pmatrix}=
\begin{pmatrix}
f(|z|+\epsilon) & f(z)\\
\overline{f({z})} & f(|z|)
\end{pmatrix}
\]
is positive definite, a contradiction. This proves $f$ is injective on $\Omega \setminus \R$. Since $f$ maps $\Omega \cap \R$ into itself and is injective on that set, it follows that $f$ is injective on $\Omega$.

(The positive semidefinite case) $(\Longleftarrow)$ The proof is similar to the positive definite case.

$(\Longrightarrow)$ This is the proof of Theorem~\ref{main_thm_1} (positive semidefinite case) appended by the following paragraphs using equation~\eqref{test-matrices-1}.

Suppose $z\in \Omega\setminus \R$ is such that $f(z)\in \R$. Then $f(\overline{z}) = \overline{f(z)} = f(z)\in \R$. If $f(z)>0$, then $f(z)=\alpha|z|^\beta$. It follows that the matrix $C:=C(z,\overline{z},z,0)$ is not positive semidefinite, but $f[C]$ is, a contradiction. On the other hand, if $f(z)<0$ then $f(z)=-\alpha|z|^{\beta}$, in which case consider $A:=A(z,z,0)\in \M_2(\Omega)$. The matrix $A$ is not positive semidefinite, but $f[A]$ is, a contradiction. Therefore $f(z)\in \C\setminus\R$.

The proof of the injectivity of $f$ over $\Omega\setminus\R$ is similar to the proof of the positive definite case by considering the matrix $A(z,\overline{w},0)\in \M_2(\Omega)$.
\end{proof}

\begin{remark}[A collection of ``irregular'' $2\times 2$ complex sign preservers and positivity preservers]\label{rem:2by2-complex}

Note that, for the $2\times 2$ preservers in Theorems~\ref{main_thm_1} and \ref{main_thm_1-oldcase}, the argument of $f$ is mostly irrelevant outside the real axis. Thus, they may not be very nice; we provide a family of such examples. Denote by $S^1_+$ the unit circle in the strict upper half plane of $\C$. Suppose $\gamma : S^1_+ \to (0,\pi)$ is injective. Define, for all $z\in S^1\setminus\{\pm 1\}$,
\begin{align*}
\theta(z):=
    \begin{cases}
    \gamma(z) & \mbox{if }\Im(z)>0,\\
    -\gamma(\overline{z}) & \mbox{if }\Im(z)<0.
    \end{cases}
\end{align*}
Fix real numbers $\alpha,\beta>0$, and define $f:\C\to \C$ by
\begin{align*}
f(z):=
\begin{cases}
\alpha|z|^{\beta} \exp({i \theta \left(z/|z|\right)}) & \mbox{if }z\in \C\setminus \R,\\
\alpha\sgn(z)|z|^{\beta} & \mbox{if }z\in \R.
\end{cases}
\end{align*}
Then, it is not difficult to check that $f[A]$ is positive definite if and only if $A\in \M_2(\C)$ is positive definite as it satisfies the conditions in Theorems~\ref{main_thm_1} and \ref{main_thm_1-oldcase}. In particular, since $\gamma$ is injective with values in $(0,\pi)$,
we have $f(z)\in\R$ if and only if $z\in\R$. Moreover, the map
$x\mapsto \alpha \sgn(x)|x|^\beta$ is injective on $\R$. Now suppose
$z,w\in\C\setminus\R$ and $f(z)=f(w)$. Then $|z|=|w|$ and
$e^{i\theta(z/|z|)}=e^{i\theta(w/|w|)}$. By construction of $\theta$,
this forces $z$ and $w$ to lie in the same half-plane, and then the
injectivity of $\gamma$ implies $z=w$. Hence $f(w)=f(z)$ if and only if
$w=z$. Finally, to obtain an ``irregular'' preserver over an $\Omega$ in Theorems~\ref{main_thm_1} and \ref{main_thm_1-oldcase} (and Propositions~\ref{2by2preservers} and \ref{2by2preservers:psd}), consider $f|_{\Omega}$.
\end{remark}

The irregular behavior of the functions $f$ described in Remark~\ref{rem:2by2-complex} stems primarily from their irregularity on \( S^1 \). Motivated by this observation, we demonstrate that, under the additional assumptions of multiplicativity and Lebesgue measurability, the class of preservers restricted to $S^1$ coincides with the continuous field automorphisms, whose classification can be found in \cite{guillot2014fractional}.

\begin{theorem}
Let \(1 < \rho \leq \infty \), and \( f : D(0,\rho) \to \mathbb{C} \) be a multiplicative function on \( D(0,\rho) \), continuous on \( S^1 \), and Lebesgue measurable on some interval \( I \subseteq (0,\rho) \) that contains $1$. The following are equivalent:
\begin{enumerate}
    \item \( f[A] \) is positive (semi)definite if and only if \(A\in \M_2(D(0,\rho)) \) is positive (semi)definite.
    \item There exists real \( \beta >0 \) and \( \kappa \in \{ \pm 1 \} \) such that
    \[
    f(z) = f(r \exp({i\theta})) = r^{\beta} \exp({i \kappa \theta})
    \]
    for all \( z \in D(0,\rho) \), where \( r = |z| \) and \( \theta = \arg(z) \).
\end{enumerate}
\end{theorem}

\begin{proof}
This follows directly from \cite[Lemma~3.1]{guillot2014fractional} and Theorem~\ref{main_thm_1-oldcase}.
\end{proof} 

\section{Real sign preservers for $n\geq 3$}\label{Sreal3}

Before proving Theorem~\ref{main_thm_2}, we first classify the power functions that are sign preservers. Henceforth, we use ${\bf 1}_{n\times n}$ to denote the $n\times n$ matrix with all entries $1$. Similarly, ${\bf 1}_{n}\in \R^n$ denotes the vector with all coordinates equal to $1$.

\begin{lemma}\label{lemma:sign_power_real}
Fix an integer $n\geq 3$, and let $\Omega\subseteq \R$ be such that $I:=\Omega\cap (0,\infty)$ is a nonempty interval.  Suppose $\beta\in (0,\infty)$, and let $f(x) := \sgn(x)|x|^{\beta}$ for all $x\in \Omega$. Assume a Hermitian matrix $A\in M_n(\Omega)$ is positive (semi)definite if and only if $f[A]$ is positive (semi)definite. Then $f(x) = x$, i.e., $\beta = 1$.
\end{lemma}
\begin{proof}
Suppose $0<a<b$ such that $(a,b)\subseteq I$ is properly contained. We first address the positive definite case. Firstly, we assume that $\beta \in (0, n-2) \setminus \N$. Consider $A(\epsilon) := \begin{pmatrix}1+\epsilon ij\end{pmatrix}_{i,j=1}^n$. By Theorem \ref{thm:fitz_horn_fractional}, $f[A(\epsilon)]$ is not positive semidefinite for all $\epsilon \in (0, \epsilon_0)$ where $\epsilon_0 > 0$ is small enough. Without loss of generality, assume $\epsilon_0$ is small enough so that $1 < 1+\epsilon_0 n^2 < b/a$ and that $a A(\epsilon)$ has positive entries in $I$. Thus, the matrix $f[a A(\epsilon)]=a^{\beta} f[A(\epsilon)]$ is not positive definite for all $\epsilon \in (0, \epsilon_0)$. For $\gamma > 0$ small enough, the matrix $a A(\epsilon) + \gamma I_n$ also has positive entries in $I$ and, by continuity of $f$, the matrix $f[a A(\epsilon) + \gamma I_n]$ is not positive definite. Thus, for $\epsilon, \gamma > 0$ small enough, the matrix $a A(\epsilon) + \gamma I_n$ is positive definite, but $f[a A(\epsilon) + \gamma I_n]$ is not positive definite. This contradicts the assumption that $f$ is a positive definite sign preserver. We must therefore have $\beta \in \N \cup [n-2, \infty)$. 

Next, assume $\beta = k \in \N \setminus \{1\}$. Let $v = (v_1,\dots,v_n)^T \in (0,\infty)^n$ be a column vector with distinct entries. Using the classical Vandermonde determinant, the vectors $v^{\circ 0} = {\bf 1}_{n}, v, v^{\circ 2}, \dots, v^{\circ (n-1)}$ are linearly independent. Let $\epsilon > 0$ be  small enough such that $1 < 1 + \epsilon (n-2) \left(\max_{1\leq i \leq n} v_i\right)^{2(n-2)} < b/a$. Consider the matrix 
\[
A = a \cdot \left({\bf 1}_{n \times n} + \epsilon \sum_{i=1}^{n-2} v^{\circ i} (v^{\circ i})^T\right). 
\]
By the choice of $\epsilon$, the matrix $A$ has positive entries in $I$. Since $A$ is the sum of $(n-1)$ many rank-one positive semidefinite matrices, $A$ is positive semidefinite and has rank $n-1$. Expanding $f[A]$ yields (up to a positive constant) all terms of the form
\[
\Big{(}v^{\circ i_1}(v^{\circ i_1})^T\Big{)} \circ \Big{(} v^{\circ i_2}(v^{\circ i_2})^T \Big{)} \circ \dots \circ \Big{(} v^{\circ i_k}(v^{\circ i_k})^T \Big{)}
\]
with $0 \leq i_j \leq n-2$ for $j=1, \dots, k$. Each such term is positive semidefinite. Moreover, since
\[
(v^{\circ i}(v^{\circ i})^T)\circ (v^{\circ j}(v^{\circ j})^T ) = v^{\circ (i+j)}(v^{\circ (i+j)})^T,
\]
the sum contains (up to a positive constant) all terms of the form $v^{\circ 0} (v^{\circ 0})^T, v^{\circ 1} (v^{\circ 1})^T, \dots, v^{\circ (n-1)} (v^{\circ (n-1)})^T$. Using again the linear independence of the entrywise powers of $v$, it follows that $f[A]$ is positive definite. Thus, $A$ is singular, but $f[A]$ is positive definite, contradicting the fact that $f$ is a sign preserver. 

Finally, assume $\beta \in (n-2, \infty) \setminus \N$. Consider the matrix $A = a \begin{pmatrix}1 + \epsilon^2 v_i v_j\end{pmatrix}_{i,j=1}^n$ with $v$ as above, and where $\epsilon > 0$ is small enough such that $1 < 1 + \epsilon^2 \max_{1\leq i \leq n} v_i^2 < b/a$. Then $A$ has positive entries in $I$, is singular, and by \cite[Theorem 6(i)]{jain2020hadamard}, the matrix $f[A]$ is positive definite. This again contradicts the assumption that $f$ is a sign preserver. Therefore, $\beta$ has to be equal to $1$. 

We now address the positive semidefinite case. When $\beta \in (0,n-2)\setminus \N$, the same argument as above provides an example of a positive semidefinite matrix $A$ such that $f[A]$ is not positive semidefinite. Next suppose $\beta \in \N \setminus \{1\}$ or $\beta \in (n-2, \infty)\setminus \N$. In each case, the proof of the positive definite case provides a singular positive semidefinite matrix $A$ such that $f[A]$ is positive definite. For $\eta > 0$, consider $A(\eta) = A-\eta I_n$. Notice that $A(\eta)$ is not positive semidefinite for all $\eta > 0$ since $A$ is singular. However, for $\eta > 0$ small enough, by continuity, the matrix $f[A(\eta)]$ is positive definite. Therefore, $\beta$ has to be equal to $1$. This concludes the proof.
\end{proof}

We can now prove Theorem~\ref{main_thm_2}.

\begin{proof}[Proof of Theorem~\ref{main_thm_2}]  $(\Longleftarrow)$ Clearly holds. 

$(\Longrightarrow)$ We first address the positive definite case. Suppose $f$ is a positive definite sign preserver on $M_n(\Omega)$. Let $\upsilon$ denote the upper-end point of $\Omega$. We begin by considering the case where $\upsilon \not\in \Omega$. Consider a positive definite matrix $A \in M_2(\Omega)$. By Lemma \ref{LPDextension}, there exists a positive definite matrix $A' \in M_n(\Omega)$ whose $2\times 2$ leading matrix is $A$. By assumption, $f[A']$ is positive definite. Consequently, $f[A]$ is positive definite. Thus, $f$ preserves positive definiteness on $M_2(\Omega)$. By Proposition~\ref{2by2preservers}, the function $f$ is positive, increasing, and continuous over the positive points in $\Omega$. 

By the same argument as above, $f$ preserves positive definiteness on $M_{n-1}(\Omega)$. Let $A \in M_{n-1}(\Omega)$ be a symmetric matrix such that $A$ is not positive definite. Consider the matrix $\E(A;x) \in M_n(\Omega)$ from Equation (\ref{eqn_extension_matrix}) where $x \in \Omega$ with $x > a_{n-1,n-1}$. Since $A$ is not positive definite, the matrix $\E(A;x)$ is also not positive definite. By assumption $f[\E(A;x)]$ is not positive definite. Suppose for a contradiction that $f[A]$ is positive definite. Then $\det f[\E(A;x)] = (f(x)-f(a_{n-1,n-1})) \det f[A] > 0$ since $f$ is increasing on $I \setminus \{0\}$. This shows that $f[\E(A;x)]$ is positive definite, a contradiction. Thus, we conclude that $f[A]$ is not positive definite. Hence, $f$ is a positive definite sign preserver on $M_{n-1}(\Omega)$. Applying this argument recursively, we obtain that $f$ is a positive definite sign preserver on $M_2(\Omega)$. Therefore, from  Theorem~\ref{main_thm_1}, there exist $\alpha,\beta >0$ such that $f(x) = \alpha\sgn(x){|x|}^{\beta}$ for all $x\in \Omega$. By Lemma~\ref{lemma:sign_power_real}, we have $\beta =1$, i.e., $f(x) = \alpha x$ for all $x\in \Omega$. 

Next, we assume that $\upsilon \in \Omega$. Using the above argument restricted to matrices with entries in $\Omega \setminus \{\upsilon\}$, we conclude that there exists $\alpha > 0$ such that $f(x) = \alpha x$ for all $x \in \Omega \setminus \{\upsilon\}$. Replacing $f$ by $f/\alpha$, we assume without loss of generality that $\alpha = 1$. We will show that $f(\upsilon) = \upsilon$. 

For $1\leq j \leq n$, choose $x_j \in \Omega$ such that $0<x_1<x_2<\ldots<x_n=\upsilon$. Let $A_1=(x_1)$ and inductively define $A_j=\E(A_{j-1};x_j)$ for $2\leq j \leq n$. Then we have $\det A_j= (x_j-x_{j-1})\det A_{j-1}$ for $2\leq j \leq n$ and thus $A_n\in M_n(\Omega)$ is positive definite. It follows that $f[A_n]$ is positive definite and thus $f(\upsilon)=f(x_n)>f(x_{n-1})=x_{n-1}$. Thus, $f(\upsilon)\geq \lim_{x_{n-1} \to \upsilon} x_{n-1}=\upsilon > 0$. 

Let $\epsilon_0$ be small enough such that for all $\epsilon \in (0,\epsilon_0)$ the matrix
$
B_2(\epsilon) = \begin{pmatrix}
\upsilon & \upsilon-\epsilon \\
\upsilon-\epsilon & \upsilon-2\epsilon
\end{pmatrix} \in M_2(\Omega).
$
Note that $B_2(\epsilon)$ is not positive definite since $\det B_2(\epsilon)=-\epsilon^2<0$. Choose $y_2, \dots, y_n \in \Omega$ such that $\upsilon-2\epsilon=y_2<y_3<\ldots<y_n<\upsilon$. We define recursively $B_{j+1}(\epsilon)=\E(B_j(\epsilon); y_{j+1})$ for each $2\leq j \leq n-1$. Note that for each $2\leq j \leq n-1$, we have $\det B_{j+1}(\epsilon)=(y_{j+1}-y_j)\det B_j(\epsilon)$ and $\det f[B_{j+1}(\epsilon)]=(y_{j+1}-y_j)\det f[B_j(\epsilon)]$. If $f[B_2(\epsilon)]$ is positive definite, then $\det f[B_j(\epsilon)]>0$ for each $2\leq j \leq n$. Thus, $f[B_n(\epsilon)]$ is positive definite; however, $B_n(\epsilon)\in M_n(\Omega)$ is not positive definite, contradicting the assumption that $f$ is a positive definite sign preserver. Thus, $f[B_2(\epsilon)]$ is not positive definite. Since $f(\upsilon)>0$, we conclude that $f(\upsilon)(\upsilon-2\epsilon)-(\upsilon-\epsilon)^2\leq 0$. Letting $\epsilon \to 0^+$, we conclude that $f(\upsilon)\leq \upsilon$. This concludes the proof of the positive definite case.

Now, assume $f$ is a positive semidefinite sign preserver. There is no need to consider whether $\upsilon\in \Omega$ because of Lemma~\ref{LPSDextension}. The same argument showing $f(x) = \alpha x$ in the positive definite case applies, with Lemma~\ref{LPDextension} replaced by Lemma~\ref{LPSDextension}, and Proposition~\ref{2by2preservers} replaced by Proposition~\ref{2by2preservers:psd}.
\end{proof}

\section{Complex sign preservers for $n\geq 3$}\label{Scomplex3} 

In this section, we prove Theorem~\ref{main_thm_3}.

\begin{proof}[Proof of Theorem~\ref{main_thm_3}] 
$(\Longleftarrow)$ Clearly holds. 

$(\Longrightarrow)$
We have that $I\subseteq [0,\infty)$ is an interval such that the annulus $\Omega\subseteq \C$ is of psd/pd-type. In both cases, i.e., psd/pd-type cases, as in the proof of Theorem~\ref{main_thm_2}, using an embedding argument with Lemma~\ref{LPSDextension}/Lemma~\ref{LPDextension}, $f$ is a psd/pd sign preserver on $M_2(\Omega)$ and $M_3(\Omega)$. So it suffices to prove the theorem for $n=3$. By Theorem~\ref{main_thm_1} and Theorem~\ref{main_thm_2}, there exists $\mu>0$ such that we have $f(x)=\mu x$ for $x\in I$ and $|f(z)|=\mu|z|$ for all $z\in \Omega$. We re-scale both, $\Omega$ and $f$, by appropriate positive scalars so that without loss of generality, the unit circle $S^1\subset \Omega$, $1 \in I$ is an interior point of $I$, and $f(1)=1$. It suffices to show that either $f\equiv z$ or $f\equiv \overline{z}$.

\textit{(The positive semidefinite case)} 
Consider the following test matrix
\begin{equation}\label{eq:A}
A:=A(a,b,c,x,y,z)= \begin{pmatrix}
a & x & y\\
\overline{x} &b & z\\
\overline{y} &\overline{z} & c
\end{pmatrix}
\end{equation}
with $a,b,c\in I$ and $x,y,z\in \Omega$ such that $ab\geq |x|^2, bc\geq |z|^2, ac \geq |y|^2$, so that all principal minors of $A$ except $\det A$ are nonnegative. Write $x=pX, y=qY, z=rZ$, where  $p,q,r \in I$ and $|X|=|Y|=|Z|=1$, with $p^2 \leq  ab$, $q^2 \leq ac$, $r^2 \leq bc$. Then $A$ is positive semidefinite if and only if
\begin{align*}
\det A
&=a(bc-|z|^2)-x(c\overline{x}-\overline{y}z)+y(\overline{x}\overline{z}-b\overline{y})\\
&=abc-a|z|^2-b|y|^2-c|x|^2+2\Re(x\overline{y}z)\\
&=abc-ar^2-bq^2-cp^2+2pqr\Re(X\overline{Y}Z)\geq 0.
\end{align*}
Note that all principal minors of $f[A]$ except $\det f[A]$ are nonnegative by assumption. Now, $f[A]$ is positive semidefinite if and only if
\begin{align*}
\det f[A]
&=a(bc-f(z)f(\overline{z}))-f(x)(cf(\overline{x})-f(\overline{y})f(z))+f(y)(f(\overline{x})f(\overline{z})-bf(\overline{y}))\\
&=abc-a|z|^2-b|y|^2-c|x|^2+2\Re(f(x)f(\overline{y})f(z))\geq 0.
\end{align*}

First, take $a=b=c=p=q=r=1$. Then $A$ is positive semidefinite if and only if $\det A=-2+2\Re(X\overline{Y}Z)\geq 0$ if and only if $X\overline{Y}Z=1$. And $f[A]$ is positive semidefinite if and only if $\det f[A]=-2+2\Re(f(X)f(\overline{Y})f(Z)) \geq 0$ if and only if $f(X)f(\overline{Y})f(Z)=1$. Thus, we have $X\overline{Y}Z=1$ (that is, $Y=XZ$) if and only if $f(X)f(\overline{Y})f(Z)=1$ (that is, $f(X)f(Z)=f(Y)$). It follows that $f$ is injective and multiplicative on the unit circle, that is, we have $f(XY)=f(X)f(Y)$ whenever $X,Y\in S^1$.

Now pick arbitrary $a=b=c=p\geq q=r\in I$, with $r > 0$. Then $A$ is positive semidefinite if and only if $\det A=-2pq^2+2pq^2\Re(X\overline{Y}Z)\geq 0$ if and only if $X\overline{Y}Z=1$. And $f[A]$ is positive semidefinite if and only if $\det f[A]=-2pq^2+2\Re(f(pX)f(q\overline{Y})f(qZ)) \geq 0$ if and only if $f(pX)f(q\overline{Y})f(qZ)=pq^2$. Setting $Z=1$, we get $X=Y$ if and only if $f(pX)f(q\overline{Y})=pq$. Thus, since $\frac{1}{p}f(pX),\frac{1}{q}f(q\overline{X})\in S^1$, and $f$ is multiplicative on the unit circle, it follows that
\[
\frac{1}{p}f(pX)\frac{1}{q}f(q\overline{X})=1 \quad \iff \quad \frac{1}{p}f(pX) =  \frac{1}{q}f(qX) \qquad \mbox{for all }X\in S^1,~p\geq q\in I.
\]
In particular, $f(rX)=rf(X)$ whenever $X\in S^1$ and $r \in I$.


Next, we take $a=b=c=1$. Fix $0<\eta<1$ such that $[\eta,1] \subset I$, and consider general numbers $p,q,r \in [\eta,1]$. Note that $A$ is positive semidefinite if and only if $\det A=1-p^2-q^2-r^2+2pqr\Re(X\overline{Y}Z)$, and $f[A]$ is positive semidefinite if and only if 
\begin{align*}
\det f[A]
&=1-p^2-q^2-r^2+2\Re(f(x)f(\overline{y})f(z))\\
&=1-p^2-q^2-r^2+2pqr \Re(f(X)f(\overline{Y})f(Z))\\
&=1-p^2-q^2-r^2+2pqr\Re(f(X\overline{Y}Z))\geq 0.
\end{align*}
Setting $X = Y = 1$, we have thus proved that for each $Z$ on the unit circle, and any real numbers $p,q,r \in [\eta,1]$, we have
\begin{equation*}
1-p^2-q^2-r^2+2pqr \Re(Z)\geq 0 \iff 1-p^2-q^2-r^2+2pqr \Re(f(Z))\geq 0.
\end{equation*}
Consider the continuous function
$$
g(p,q,r)=\frac{p^2+q^2+r^2-1}{2pqr},
$$
where $p,q,r \in [\eta,1]$. Since $0<\eta<1$, we have
\[
\alpha := \min_{p,q,r \in [\eta, 1]} g(p,q,r) \leq g(\eta, \eta, \eta) = \frac{3\eta^2-1}{2\eta^3} < 1.
\]
Note that $g(1,1,1)=1$. Thus, by the continuity of $g$, for each $t \in (\alpha,1]$, we can find $p,q,r\in  [\eta,1]$ such that $g(p,q,r)=t$.


Fix $Y \in \C$ such that $|Y| = 1$ and let $(X_n)_{n \geq 1} \subseteq \C$ be such that $|X_n| = 1$ and $X_n \to Y$, i.e., $X_n \overline{Y} \to 1$ as $n \to \infty$. In particular, $\Re(X_n \overline{Y}) > \alpha$ for all $n$ large enough, as $\alpha<1$. Fix any such $n$ sufficiently large; we claim that $\Re(f(X_n) \overline{f(Y)})\geq \Re(X_n \overline{Y})$. Suppose otherwise that $\Re(X_n \overline{Y})>\Re(f(X_n \overline{Y}))$. Since $\Re(X_n \overline{Y})\in (\alpha,1]$, we can choose $p,q,r \in [\eta,1]$ such that $g(p,q,r)=\Re(X_n \overline{Y})$, that is, $1-p^2-q^2-r^2+2pqr \Re(X_n \overline{Y})=0$; however, since $\Re(X_n \overline{Y})>\Re(f(X_n \overline{Y}))$ we have $1-p^2-q^2-r^2+2pqr \Re(f(X_n \overline{Y}))<0$, a contradiction. This proves the claim and consequently we have $\Re(f(X_n) \overline{f(Y)}) \to 1$ as $n \to \infty$, that is, $f(X_n) \to f(Y)$ as $n \to \infty$, as required. 

We have thus shown that the restriction of $f$ to the unit circle $\{z\in \C: |z|=1\}$ is continuous and injective. Hence, $f$ is an injective homomorphism of the unit circle into itself. Recall that the continuous homomorphisms of the unit circle group are all integer power functions (see e.g.~\cite[Proposition 7.1.1]{deitmar2005first}). We conclude that $f(z) \equiv z$ on the unit circle or $f(z) \equiv \overline{z}$ on the unit circle since $f$ needs to be injective. Since $f(rX)=rf(X)$ whenever $X$ is on the unit circle and $r \in I$, we conclude that $f(z) \equiv z$ on $\Omega$ or $f(z) \equiv \overline{z}$ on $\Omega$.

\textit{(The positive definite case)} The proof of the positive definite case is very similar with the help of a limiting argument. Let $f$ be a positive definite sign preserver over $M_3(\Omega)$. Following the proof above, we can show that we have $f(x)=x$ for $x\in I$ and $|f(z)|=|z|$ for all $z\in \Omega$. 



Here we point out the difference in showing $f$ is multiplicative on the unit circle. Let $\epsilon>0$ such that $1+\epsilon\in I$. Consider the same matrix $A$ as defined in equation~\eqref{eq:A} with $a,b,c=1+\epsilon$, $p,q,r=1$, and $X,Y,Z$ on the unit circle with $Y=XZ$, so that
$$
\det A=(1+\epsilon)^3-3(1+\epsilon)+2\Re(X\overline{Y}Z)=-2+3\epsilon^2+\epsilon^3+2\Re(X\overline{Y}Z)=3\epsilon^2+\epsilon^3>0
$$
and $\det f[A]=-2+3\epsilon^2+\epsilon^3+2\Re(f(X)\overline{f(Y)}f(Z))$. Note that $A$ and $f[A]$ are both positive definite. Letting $\epsilon \to 0$, we can deduce $f(X)\overline{f(Y)}f(Z)=1$, that is, $f$ is multiplicative on the unit circle. Following an almost identical argument as in the positive semidefinite case, we obtain that $f(z) \equiv z$ on $\Omega$ or $f(z) \equiv \overline{z}$ on $\Omega$.
\end{proof}

\begin{remark}\label{remarkC}
We remark that the above proof works for more general domains. For example, it can be easily adapted to domains that are given by polar rectangles of the form $\Omega=\{re^{i\theta}: r \in I, \theta \in J\}$, where $I,J$ are intervals with $0\in J$. However, Theorem \ref{main_thm_3} does not extend to all complex pd/psd-type subsets $\Omega$, as we illustrate below.

Let $k\geq 2$ be an integer and let $\epsilon>0$ be sufficiently small. Let $\Omega=(1-\epsilon, 1+\epsilon] \cup T$, where $T=\{z_1, \overline{z_1}, \ldots, z_k, \overline{z_k}\}$ is a subset of the unit circle, with $z_j=\exp(i \pi \cdot 5^{j-k-2})$ for $1\leq j \leq k$. Then $\Omega$ is of psd type. Let $f:\Omega \to \C$ be a function such that $f(x)=x$ for $x\in (1-\epsilon,1+\epsilon]$,  and for each $1\leq j \leq k$, $f(z_j)\in \{z_j,\overline{z_j}\}$ and $f(\overline{z_j})=\overline{f(z_j)}$. We claim that $f$ is a psd sign preserver on $\M_3(\Omega)$. Let $A$ be a Hermitian matrix in $\M_3(\Omega)$ as defined in equation~\eqref{eq:A}, where $a,b,c\in (1-\epsilon,1+\epsilon]$, and $x,y,z\in \Omega$. The function $f$ is a psd sign preserver on $\M_2(\Omega)$. Thus, it suffices to show $\det A\geq 0$ if and only if $\det f[A]\geq 0$. Similar to the computation above, we have
\begin{align*}
\det A&=abc-a|z|^2-b|y|^2-c|x|^2+2\Re(x\bar{y}z),\text{ and}\\
\det f[A]&=abc-a|z|^2-b|y|^2-c|x|^2+2\Re(f(x)\overline{f(y)}f(z)).
\end{align*}
By the definition of $z_j$'s, if $x\overline{y}z$ is real, then the product of two of $x,\overline{y},z$ has to be real. If the product of two of $x,\overline{y},z$ is real, then clearly $\Re(x\bar{y}z)=\Re(f(x)\overline{f(y)}f(z))$, and thus $\det A \geq 0$ if and only if $\det f[A]\geq 0$. Next, assume that $x\overline{y}z$ is not real and that the product of any two of $x,\overline{y},z$ is non-real; in particular, at most one of $x,y,z$ is real. We claim that in this case, when $\epsilon$ is sufficiently small, we always have $\det A<0$ and $\det f[A]<0$. Indeed, $\Re(x\overline{y}z)$ and $\Re(f(x)\overline{f(y)}f(z))$ are both at most $(1+\epsilon) \cos (\pi/5^{k+1})$, while $abc-a|z|^2-b|y|^2-c|x|^2\leq (1+\epsilon)^3-3(1-\epsilon)^3=-2+O(\epsilon)$.

\end{remark}

\begin{remark}\label{rem:M_n-case}
A natural question arises when $n\geq 3$, akin to the one addressed in Theorem~\ref{main_thm_1-oldcase}: what are the corresponding ``sign preservers'' over \textit{all} matrices in $M_n(\Omega)$? Our proofs in Section~\ref{Sreal3} and this section reveal that as in Theorem~\ref{main_thm_2} and Theorem~\ref{main_thm_3}, the set of these coincide with exactly the positive multiples of continuous field automorphisms.
\end{remark}

\section{Sign preservers for graphs}\label{sec:graphs}
This section is dedicated to proving Theorem~\ref{TsignG}. First, we introduce some basic terminology from graph theory. We denote the path graph, the cycle graph, and the complete graph on $n$ vertices by $P_n$, $C_n$, and $K_n$ respectively. Given a graph $G = (V,E)$, a subgraph of $G$ is a graph $G' = (V',E')$ where $V' \subseteq V$ and $E' \subseteq E \cap (V' \times V')$. Given a subset $V' \subseteq V$, the subgraph of $G$ induced by $V'$ is the subgraph $G' = (V', E')$ where $E' = E \cap (V' \times V')$. We begin by showing that positivity preservers on $M_G$ are also positivity preservers on $M_{G'}$ for any induced subgraph $G'$ of $G$. This is obvious when $0 \in \Omega$, but is non-trivial when $0$ is not in the domain.
\begin{lemma}\label{Lextension}
Let $G$ be a connected graph on $n \geq 3$ vertices and let $f: \Omega \to \C$ where $\Omega$ is any of $\R, \C$, or $(0,\infty)$. If $f_G$ is a positive definite preserver on $M_G(\Omega)$, then $f_{G'}$ is a positive definite preserver on $M_{G'}(\Omega)$ for all induced subgraphs $G'$ of $G$. The same holds if ``definite'' is replaced by ``semidefinite''.     
\end{lemma}
\begin{proof}
We assume $V' =\{1,\dots,k\}$ with $1 \leq k < n$, and we let $G'$ be the induced subgraph of $G$ by $V'$. First, assume that $f_G$ is a positive definite preserver on $M_G(\Omega)$ and let $A \in M_{G'}(\Omega)$ be positive definite. For $\epsilon > 0$, define $M(\epsilon) \in M_G(\Omega)$ by 
\[
M(\epsilon) := \begin{pmatrix}
A & B \\
B^* & C
\end{pmatrix}, 
\]
where $b_{ij} = \epsilon$ if $i$ and $j$ are adjacent in $G$ and zero otherwise, $c_{ii} = 1$ for all $i$, and for $i \ne j$, $c_{ij} = \epsilon$ if $i$ and $j$
 are adjacent in $G$ and $0$ otherwise. For $\epsilon$ small enough, the matrix $C$ is positive definite and the Schur complement $A - BC^{-1}B^*$ is positive definite. Thus, $M(\epsilon)$ is positive definite for $\epsilon > 0$ small enough. By assumption, $f_G[M(\epsilon)]$ is positive definite and therefore so is $f_{G'}[A]$. This proves $f_{G'}$ is a positive definite preserver on $M_{G'}(\Omega)$. 


Now, suppose $f_G$ is a positive semidefinite preserver on $M_G(\Omega)$. Since $G$ is connected, it has a copy of $K_2$, and so it follows that for any positive definite $A \in M_2((0,\infty))$, the matrix $f[A]$ is positive semidefinite. By Corollary \ref{Cpdtopsd}, the function $f$ is continuous on $(0, \infty)$. Now, let $A \in M_{G'}(\Omega)$ be positive semidefinite. Using the same argument as in the positive definite part of the proof, for $\eta > 0$, extend the matrix $A + \eta I$ to a positive definite matrix $M \in M_G(\Omega)$. By assumption $f_G[M]$ is positive semidefinite and it follows that $f_{G'}[A + \eta I]$ is positive semidefinite for every $\eta > 0$. Letting $\eta \to 0^+$ and using the continuity of $f$ on $(0, \infty)$, we conclude that $f_{G'}[A]$ is positive semidefinite. This proves $f_{G'}$ is a positive semidefinite preserver on $M_{G'}(\Omega)$.  
 \end{proof}

We next show that when $G$ has at least $3$ vertices, the non-constant positivity preservers on $M_G((0,\infty))$ are necessarily non-bounded.
\begin{lemma}\label{Lunbounded}
Let $f: (0,\infty) \to \R$ with $f$ non-constant. Let $G$ be a connected graph on $n \geq 3$ vertices. Suppose $f_G[A]$ is positive semidefinite for all positive semidefinite $A \in M_G((0,\infty))$. Then $\lim_{x \to \infty} f(x) = \infty$. The same holds if ``semidefinite'' is replaced by ``definite''.  
\end{lemma}
\begin{proof}
Since $G$ is connected, it contains a copy of $K_2$. Thus, Lemma~\ref{Lextension} implies that $f$ preserves positive semidefiniteness on $M_2((0,\infty))$. In particular, by Proposition \ref{2by2preservers:psd}, the function $f$ is continuous, nondecreasing and never zero. Suppose first $G$ does not contain $K_3$ as an induced subgraph. Then it contains $P_3$ as an induced subgraph and Lemma~\ref{Lextension} implies that $f_{P_3}$ preserves positive semidefiniteness on $M_{P_3}((0,\infty))$. For $n > 0$, consider the positive definite matrix 
\[
A = \begin{pmatrix}
3n & n & n \\
n & n & 0 \\
n & 0 & n
\end{pmatrix}.
\]
Then $f[A]$ is positive semidefinite and so 
$
\det f_{P_3}[A] = f(n)^2 (f(3n)-2f(n)) \geq 0. 
$
Since $f(n) \ne 0$, we conclude that $f(3n) \geq 2 f(n)$. It follows that $f(3^k) \geq 2^k f(1)$ for all $k \geq 1$. Since $f(1)>0$, we obtain $\lim_{x \to \infty} f(x) = \infty$. 

Now, suppose $G$ contains $K_3$ as an induced subgraph. Then Lemma~\ref{Lextension} implies that $f$ preserves positivity on $M_3((0,\infty))$. By \cite[Theorem 1.2]{horn1969theory}, the function $f$ is convex. Since it is also non-constant, it follows that $f$ is unbounded.

To prove the positive definite case, first observe that by Proposition \ref{2by2preservers}, the function $f$ is continuous on $(0,\infty)$. Approximating positive semidefinite matrices $A \in M_G((0,\infty))$ by $A + \epsilon I$ for $\epsilon > 0$, we obtain that $f_G$ preserves positive semidefiniteness on $M_G((0,\infty))$, and the result follows.  
\end{proof}

Our next result shows how induced subgraphs inherit sign preservers. 
\begin{proposition}\label{Tinduced}
Let $G$ be a connected graph on $n \geq 3$ vertices and let $f: \Omega \to \C$ be non-constant, where $\Omega$ is any of $\R, \C$, or $(0,\infty)$. Then the following are equivalent: 
\begin{enumerate}
\item $f_G$ is a positive definite sign preserver on $M_G(\Omega)$. 
\item $f_{G'}$ is a positive definite sign preserver on $M_{G'}(\Omega)$ for all induced subgraphs $G'$ of $G$. 
\end{enumerate}
The same holds if ``definite'' is replaced by ``semidefinite''. 
\end{proposition}
\begin{proof}
The $(2) \implies (1)$ implication is trivial since $G$ is an induced subgraph of itself. We prove $(1) \implies (2)$. Let $G'$ be an induced subgraph of $G$. Without loss of generality, assume $f(1) = 1$, label the vertices of $G'$ by $\{1,2,\dots,k\}$, and label the remaining vertices of $G$ by $\{k+1, \dots, n\}$. For each $1\leq i \leq n$, let $G_i$ be the subgraph of $G$ induced on $\{1,\dots i\}$. In particular, $G'=G_k$.

We first prove the positive definite case. Suppose $f_G$ is a positive definite sign preserver on $M_G(\Omega)$. Then by Lemma~\ref{Lextension}, the function $f_{G'}$ preserves positive definiteness on $M_{G'}(\Omega)$. Now, let $A_k \in M_{G'}(\Omega)$ be a Hermitian matrix that is not positive definite. We extend the matrix $A_k \in M_{G_k}(\Omega)$ to a matrix in $M_G(\Omega)$ inductively as follows. For a given $k \leq i \leq n-1$, suppose $A_i \in M_{G_i}(\Omega)$ has already been constructed. We define $A_{i+1} \in M_{G_{i+1}}(\Omega)$ by 
\begin{equation}\label{EAnp1}
A_{i+1} := \begin{pmatrix}
A_i & {\bf v}_i \\
{\bf v}_i^* & d_{i+1}
\end{pmatrix}, 
\end{equation}
where $d_{i+1} > 0$ and where ${\bf v}_i \in \R^i$ is given by
\[
({\bf v}_i)_j = \begin{cases}
1 & \textrm{if } (j, i+1) \in E, \\
0 & \textrm{otherwise}. 
\end{cases}
\]
Using Schur complements, the matrix $A_{i+1}$ is positive definite if and only if $A_i - \frac{1}{d_{i+1}} {\bf v}_i {\bf v}_i^*$ is positive definite. Similarly, $f_{G_{i+1}}[A_{i+1}]$ is positive definite if and only
if $f_{G_{i}}[A_i] - \frac{1}{f(d_{i+1})}{\bf v}_i{\bf v}_i^*$ is positive definite. Now, consider the matrix $A_n \in M_G(\Omega)$. Since its leading $k \times k$ principal submatrix $A_k$ is not positive definite, the matrix $A_n$ is not positive definite for any choice of $d_{k+1}, \dots, d_n > 0$. By assumption, $f_G[A_n]$ is not positive definite. Thus, $f_{G_{n-1}}[A_{n-1}] - \frac{1}{f(d_{n})} {\bf v}_{n-1} {\bf v}_{n-1}^*$ is not positive definite. Letting $d_{n} \to \infty$ and using Lemma \ref{Lunbounded}, we conclude  that $f_{G_{n-1}}[A_{n-1}]$ is not positive definite for any choice of $d_{k+1}, \dots, d_{n-1} > 0$. Recursively proceeding in the same manner, we conclude that $f_{G'}[A_k]$ is not positive definite. This proves $f_{G'}$ is a positive definite sign preserver on $M_{G'}(\Omega)$.

Next, we address the positive semidefinite case. Suppose now $f_G$ is a positive semidefinite sign preserver on $M_G(\Omega)$. Since $G$ contains a copy of $K_2$, the function $f$ preserves positive semidefiniteness on $M_{K_2}(\Omega) = M_2(\Omega)$ by Lemma~\ref{Lextension}. By Proposition \ref{2by2preservers:psd}, $f$ is non-decreasing and never $0$ on $(0,\infty)$. We first prove that $f$ is increasing on $(0,\infty)$. Suppose first that $G$ contains $P_3$ as an induced subgraph. Since $f_G$ is a sign preserver on $M_G(\Omega)$, it preserves positive semidefiniteness on $M_{P_3}(\Omega)$ by Lemma~\ref{Lextension}. For $x, y > 0$, consider the positive semidefinite matrix 
\[
A = \begin{pmatrix}
x + y & x & y \\
x & x & 0 \\
y & 0 & y
\end{pmatrix} \in M_{P_3}(\Omega).
\]
Then $f_{P_3}[A]$ is positive semidefinite and so 
$
\det f_{P_3}[A] = f(x)f(y)(f(x+y)-f(x)-f(y)) \geq 0.
$
We conclude that $f(x+y) \geq f(x) + f(y)$ and, since $f$ is never $0$, it must be increasing on $(0,\infty)$.

Now assume that $G$ does not contain $P_3$ as an induced subgraph, then $G$ contains $K_3$ as an induced subgraph. Then $f$ preserves positive semidefiniteness on $M_3(\Omega)$ by Lemma~\ref{Lextension}. By \cite[Theorem 1.2]{horn1969theory}, the function $f$ is convex. Let $x, y \in (0,\infty)$ with $x < y$ satisfy $f(x) = f(y)$. Since $f$ is non-decreasing, it must be constant on the interval $[x,y]$. For any $a < x$, write $x = \lambda a + (1-\lambda) y$ for some $0 < \lambda < 1$. By convexity, we obtain 
\[
f(x) = f(\lambda a + (1-\lambda)y) \leq \lambda f(a) + (1-\lambda) f(y) = \lambda f(a) + (1-\lambda) f(x). 
\]
This implies $f(a) \geq f(x)$ and therefore $f(a) = f(x)$ since $f$ is non-decreasing. Therefore, $f \equiv c$ on $(0,y]$. By Lemma \ref{Lunbounded}, there exists $z > y$ such that $f(z) > c$. Without loss of generality, choose such a $z$ such that $f \equiv c$ on $(0,0.7z]$. Consider the matrix 
\[
A = \begin{pmatrix}2z & z/2 & z \\
z/2 & z/2 & z/2 \\
z & z/2 & 0.7z
\end{pmatrix} = z \begin{pmatrix}
2 & 1/2 & 1 \\
1/2 & 1/2 & 1/2 \\
1 & 1/2 & 0.7
\end{pmatrix} \in M_3(\Omega). 
\]
Then $A$ is positive semidefinite and  
$
\det f[A] = -c (f(z)-c)^2 < 0. 
$
This contradicts the fact that $f$ preserves positive semidefiniteness on $M_3(\Omega)$. We therefore conclude that $f$ is increasing. 

Now, since $f$ is a positive semidefinite sign preserver on $M_G(\Omega)$, the map $f_{G'}$ is a positive semidefinite preserver on $M_{G'}(\Omega)$ by Lemma~\ref{Lextension}. Let $A_k \in M_{G'}(\Omega)$ be a Hermitian matrix that is not positive semidefinite. Proceed as in the positive definite case (Equation \eqref{EAnp1}) to construct matrices $A_{k+1}, \dots, A_n$. Since the $k \times k$ leading principal submatrix of $A_n$ is not positive semidefinite, the matrix $A_n$ is also not positive semidefinite for any choice of $d_{k+1}, \dots, d_n > 0$. Hence, $f_G[A_n]$ is not positive semidefinite. Using Schur complements, the matrix $f_{G_{n-1}}[A_{n-1}] - \frac{1}{f(d_{n})} {\bf v}_{n-1} {\bf v}_{n-1}^*$ is not positive semidefinite as well. Suppose for a contradiction that $f_{G_{n-1}}[A_{n-1}]$ is positive semidefinite. Since $A_k$ is not positive semidefinite, there exists $\epsilon_0 > 0$ such that $A_k + \epsilon_0 I_k$ is not positive semidefinite. This implies $A_n + \epsilon_0 I_n$ is not positive semidefinite, and therefore, neither are the matrices $f_G[A_n + \epsilon_0 I_n]$ and $f_{G_{n-1}}[A_{n-1} + \epsilon_0 I_{n-1}] - \frac{1}{f(d_{n} + \epsilon_0)} {\bf v}_{n-1} {\bf v}_{n-1}^*$ for any choice of $d_{k+1}, \dots, d_n > 0$. Now, since $f_{G_{n-1}}[A_{n-1}]$ is positive semidefinite and since $f$ is increasing on $(0, \infty)$, we have $f_{G_{n-1}}[A_{n-1}+\epsilon_0 I_{n-1}] = f_{G_{n-1}}[A_{n-1}] + D$, where $D$ is a diagonal matrix with positive diagonal entries. Thus $f_{G_{n-1}}[A_{n-1}+\epsilon_0 I_{n-1}]$ is positive definite. Letting $d_{n} \to \infty$ and using Lemma \ref{Lunbounded}, it follows that $f_{G_{n-1}}[A_{n-1} + \epsilon_0 I_{n-1}] - \frac{1}{f(d_{n} + \epsilon_0)} {\bf v}_{n-1} {\bf v}_{n-1}^*$ is positive definite for values of $d_{n}$ large enough, a contradiction. We therefore conclude that $f_{G_{n-1}}[A_{n-1}]$ is not positive semidefinite. Recursively proceeding in the same manner, we conclude that $f_{G'}[A_k]$ is not positive semidefinite. This proves $f_{G'}$ is a positive semidefinite sign preserver on $M_{G'}(\Omega)$. 
\end{proof}

\begin{remark}
For simplicity, Proposition \ref{Tinduced} was stated only for $\Omega = \R, \C$, or $(0,\infty)$. However, we note that the result holds with the same proof on more general domains (e.g., $\Omega = [0,\infty)$). 
\end{remark}

The following technical lemma is crucial in our proof of Theorem \ref{TsignG}. Let $n \geq 3$ and $\Lambda \in M_{C_n}(\C)$. Up to simultaneously permuting the rows and columns of $A$, we can assume it is of the form 
\begin{equation}\label{EqnA}
\Lambda = \Lambda(T, \lambda_{n,n}, \alpha, \beta) = \begin{pmatrix}
    T & v \\
    v^* & \lambda_{n,n}
\end{pmatrix}, 
\end{equation}
where $T \in M_{P_{n-1}}(\C)$ is tridiagonal, and $v = (\alpha, 0, \dots, 0, \beta)^T \in \C^{n-1}$. 

\begin{lemma}\label{Lrealpart}
Let $n \geq 3$. Let $\alpha = r_1 e^{i\theta_1}, \beta = r_2 e^{i\theta_2} \in \C$, where $r_1, r_2 > 0$ and $\theta_1, \theta_2 \in (-\pi, \pi]$. Then for any $\epsilon > 0$, there exist $T \in M_{P_{n-1}}((0,\infty))$ and $\lambda_{n,n} \in (0,\infty)$ such that the matrix $\Lambda(T, \lambda_{n,n}, \alpha, \beta)$ defined as in equation \eqref{EqnA} is positive definite and, for $a, b \in \C$ with $|a| = r_1$ and $|b| = r_2$, we have:
\begin{enumerate}
\item If $n$ is odd and $\Re(a\overline{b}) < \Re(\alpha \overline{\beta}) - \epsilon$, the matrix $\Lambda(T, \lambda_{n,n}, a, b)$ is not positive semidefinite; 
\item If $n$ is even and $\Re(a\overline{b}) > \Re(\alpha\overline{\beta}) +\epsilon$, the matrix $\Lambda(T, \lambda_{n,n}, a, b)$ is not positive semidefinite.
\end{enumerate}
\end{lemma}

\begin{proof}
Fix $R > 0$. For $k \geq 2$, let $T_k \in M_{P_k}((0,\infty))$ be the matrix with diagonal entries $2R$ and super/sub diagonal entries $R$. In other words, $T_k$ is a tridiagonal matrix with the diagonal entries $2R$ and other nonzero entries $R$. We claim that $\det T_k = (k+1) R^{k}$. This is immediately verified for $k=1, 2$. To prove the general case by induction, suppose this holds for matrices of dimensions $1, \dots, k$ for some $k \geq 2$. Expanding the determinant yields the recursion
\[
\det T_{k+1} = 2R \det T_k - R^2 \det T_{k-1} = 2R (k+1) R^k - R^2 k R^{k-1} = (k+2)R^{k+1}, 
\]
and the claim follows. 

Now, set $T := T_{n-1}$. The above calculation shows that all the leading principal minors of $T$ are positive and so $T$ is positive definite. Let
\[
\lambda_{n,n} = \frac{(n-1)(r_1^2+r_2^2)}{Rn} +  (-1)^{n}\frac{2}{Rn}\left(\Re(\alpha\overline{\beta})+(-1)^{n} \epsilon \right). 
\] 
Note that $\lambda_{n,n}>0$ since $r_1^2+r_2^2\geq 2|\Re(\alpha\overline{\beta})|$. Let $v = (a, 0, \dots, 0, b)^T$, where $a, b \in \C$ with $|a| = r_1$ and $|b| = r_2$. We claim that the resulting matrix $\Lambda(T, \lambda_{n,n}, a, b)$ has the required properties. Indeed, consider the Schur complement 
\[
S := \lambda_{n,n} - v^* T^{-1} v;
\]
since $T$ is positive definite, $\Lambda(T, \lambda_{n,n}, a,b)$ is positive definite if and only if $S>0$, and $\Lambda(T, \lambda_{n,n}, a,b)$ is positive semidefinite if and only if $S\geq 0$.

Using a calculation similar to the one for $\det T_k$, one can verify that
\[
(T^{-1})_{1,1} = (T^{-1})_{n-1,n-1} = \frac{n-1}{Rn}, \qquad \textrm{and} \qquad (T^{-1})_{1,n-1} = \frac{(-1)^{n}}{Rn}.
\]
Hence, we have 
\begin{align*}
S = \lambda_{n,n} - \frac{(n-1)(r_1^2+r_2^2)}{Rn} -  (-1)^{n} \frac{2}{Rn} \Re(a \overline{b}) = (-1)^{n} \frac{2}{Rn} \left(\Re(\alpha\overline{\beta}) - \Re(a\overline{b}) + (-1)^{n}\epsilon\right).
\end{align*}
When $(a,b) = (\alpha, \beta)$, we obtain $S > 0$ and so $\Lambda(T, \lambda_{n,n}, \alpha, \beta)$ is positive definite. When $n$ is odd, the matrix $\Lambda(T, \lambda_{n,n}, a, b)$ is positive semidefinite if and only if $\Re(\alpha\overline{\beta}) - \Re(a\overline{b}) - \epsilon \leq 0$. Thus, the matrix $\Lambda(T, \lambda_{n,n}, a, b)$ is not positive semidefinite whenever $\Re(a\overline{b})< \Re(\alpha\overline{\beta}) - \epsilon$. Similarly, when $n$ is even, the matrix $\Lambda(T, \lambda_{n,n}, a, b)$ is positive semidefinite if and only if $\Re(\alpha\overline{\beta}) - \Re(a\overline{b}) + \epsilon \geq 0$. Thus, the matrix $\Lambda(T, \lambda_{n,n}, a, b)$ is not positive semidefinite when $\Re(a\overline{b}) > \Re(\alpha\overline{\beta}) + \epsilon$. 
\end{proof}

We conclude the section by proving Theorem~\ref{TsignG}.

\begin{proof}[Proof of Theorem~\ref{TsignG}]
It is not difficult to show that $f_G$ is not a sign preserver when $f$ is constant. Hence, we assume below that $f$ is non-constant. We only prove the positive definite case. The proof of the positive semidefinite case is similar. 

(1) Suppose first $G$ is a tree and let $f_G$ be a sign preserver on $M_G(\Omega)$. Since $G$ contains $K_2$ as an induced subgraph, by Proposition \ref{Tinduced} and Theorem \ref{main_thm_1}, the sign preserver $f$ must be among the preservers in Theorem~\ref{main_thm_1} for $\alpha,\beta>0$. Without loss of generality, suppose $\alpha=1$.

Since $G$ is a tree on $n \geq 3$ vertices, it contains a copy of $P_3$ as an induced subgraph, and thus by Proposition \ref{Tinduced}, the function $f_{P_3}$ is a sign preserver on $M_{P_3}(\Omega)$. In particular, $f_{P_3}$ preserves positive definiteness on $M_{P_3}((0,\infty))$. By continuity and by extending $f$ to $[0,\infty)$ as $f(0) = 0$ when $\Omega =
(0, \infty)$, it also preserves positive semidefiniteness on $M_{P_3}([0,\infty))$. Using \cite[Theorem 2.2]{GKR-critG}, we conclude that $\beta \geq 1$. It remains to show that none of $\beta>1$ yields a sign preserver on $M_G(\Omega)$. Let $\beta > 1$. For $x \in (0,1)$, consider the matrices 
\[
B = B(x) = \begin{pmatrix}
1 & x & 0 \\
x & 1 & x \\
0 & x & 1
\end{pmatrix} \in M_{P_3}(\Omega).
\]
We have $\det B(x) = 1-2x^2$. Pick $x \in (0,1)$ such that $x^2 > 1/2$ and $x^{2\beta} < 1/2$. Then the matrix $B(x)$ is not positive semidefinite, but $f_{P_3}[B(x)]$ is positive definite. This contradicts the assumption that $f_{P_3}$ is a sign preserver. We therefore conclude that $\beta = 1$. 

Conversely, suppose (1b) holds and let $A \in M_G(\Omega)$ be a Hermitian matrix. We claim that the positive definiteness of $A$ is independent of the arguments of the off-diagonal entries of $A$. Indeed, for a given ordering of the vertices of $G$, let $G_i$ denote the subgraph of $G$ induced on $\{1,\dots, i\}$. Choose an ordering of the vertices such that $G_2 = K_2$ and, for $2 \leq i \leq n-1$, the graph $G_{i+1}$ is obtained from $G_i$ by adding a single edge $(s_i, i+1)$ to $G_i$, where $1 \leq s_i \leq i$. Let $A_i$ denote the $i \times i$ leading principal submatrix of $A$. Notice that, for $2 \leq i \leq n$, 
\[
A_i = \begin{pmatrix}
A_{i-1} & {\bf v}_{i-1} \\
{\bf v}_{i-1}^* & a_{i,i}
\end{pmatrix}, 
\]
where all the entries of ${\bf v}_{i-1}$ are $0$, except in position $s_{i-1}$. Clearly, the positive definiteness of $A_2$ does not depend on the argument of its off-diagonal entry. Suppose this is true for $A_{i-1}$ for some $3 \leq i \leq n$. Using Schur complements, the matrix $A_i$ is positive definite if and only if $a_{i,i} > 0$ and the matrix $A_{i-1} - \frac{1}{a_{i,i}}{\bf v}_{i-1}{\bf v}_{i-1}^*$ is positive definite. Observe that the entries of ${\bf v}_{i-1}{\bf v}_{i-1}^*$ only depend on the modulus of the $s_{i-1}$-th entry of ${\bf v}_{i-1}$. From this and using the inductive hypothesis, it follows that the positive definiteness of $A_i$ is independent of the arguments of its off-diagonal entries. Since $|f(z)| = \alpha |z|$ for all $z \in \Omega$, it follows immediately that $f_G$ is a positive definite sign preserver.   

(2) Suppose now $G$ is not a tree. Let $3 \leq k \leq n$ be the smallest integer such that $G$ contains a cycle $C_k$ as a subgraph. If $k=3$, then $G$ contains $K_3$ as an induced subgraph. By Proposition \ref{Tinduced} and Theorems \ref{main_thm_2} and \ref{main_thm_3}, we conclude that $f(z) \equiv \alpha z$ or $f(z) \equiv \alpha \overline{z}$ for some $\alpha > 0$. Now, suppose $k \geq 4$. We claim that $G$ contains $C_k$ as an induced subgraph. Indeed, if the subgraph induced on the vertices of $C_k$ contains a chord, then $G$ contains a cycle on less than $k$ vertices, contradicting the minimality of $k$. Therefore, by Proposition \ref{Tinduced}, the function $f_{C_k}$ is a positive definite sign preserver on $M_{C_k}(\Omega)$. Since $C_k$ contains $P_3$ as an induced subgraph, the proof of (1) shows that $f$ is among the preservers in Theorem~\ref{main_thm_1} with $\alpha>0$ and $\beta = 1$. Without loss of generality, assume $\alpha = 1$. In particular, $f(x) = x$ for all $x \in \Omega \cap \R$. We will prove that for all $z \in \Omega$, we have $\Re f(z) = \Re(z)$. Suppose for a contradiction that there exists $\gamma \in \Omega$ such that $\Re f(\gamma) \ne \Re( \gamma)$, say $|\Re f(\gamma) - \Re(\gamma)|>\epsilon$ for some $\epsilon > 0$. Suppose first $\Re f(\gamma) < \Re(\gamma) - \epsilon$ and suppose $k$ is odd. By Lemma \ref{Lrealpart}, one can choose $T \in M_{P_{k-1}}((0,\infty))$ and $a_{k,k} \in (0,\infty)$ such that the matrix $\Lambda(T, a_{k,k}, \gamma,1)$ is positive definite, but the matrix $f_{C_k}[\Lambda(T, a_{k,k},\gamma, 1)] = \Lambda(T, a_{k,k},f(\gamma), 1)$ is not positive semidefinite. If instead $\Re f(\gamma) > \Re(\gamma) + \epsilon$, apply Lemma \ref{Lrealpart} with $\gamma$ replaced by $f(\gamma)$ to obtain a matrix such that $\Lambda(T, a_{k,k}, f(\gamma), 1)$ is positive definite, but $\Lambda(T, a_{k,k}, \gamma, 1)$ is not positive semidefinite. Either case contradicts that $f_{C_k}$ is a sign preserver on $M_{C_k}(\Omega)$. The case where $n$ is even can be addressed similarly. This proves $\Re f(z) = \Re(z)$ for all $z \in \Omega$. Therefore, for each $z \in \Omega$, we must have $f(z) = z$ or $f(z) = \overline{z}$. 

Finally, assume there exist $z_1, z_2 \in \Omega \setminus \R$ such that $f(z_1) = z_1$ and $f(z_2) = \overline{z_2}$. Let $\epsilon > 0$ be such that $|\Re(z_1z_2) - \Re(z_1 \overline{z_2})| > \epsilon$. Using Lemma \ref{Lrealpart} as above, one can construct a Hermitian matrix $A \in M_{C_k}(\Omega)$ such that either $A$ is positive definite and $f_{C_k}[A]$ is not positive semidefinite, or $A$ is is not positive semidefinite and $f_{C_k}[A]$ is positive definite. This contradicts that $f_{C_k}$ is a sign preserver. We must therefore have $f(z) \equiv z$ or $f(z) \equiv \overline{z}$. This concludes the proof of $(2a) \implies (2b)$. The converse implication is trivial. 
\end{proof}

\section{Monotonicity preservers}\label{Smonotone}
Recall that given two $n \times n$ positive semidefinite matrices, we write $A \geq B$ if $A-B$ is positive semidefinite, and $A > B$ if $A-B$ is positive definite. We conclude the paper by showing how previous work on Loewner monotone maps \cite{fitzgerald1977fractional, guillot2015complete, hiai2009monotonicity} can be extended in the spirit of sign preservers. 

\begin{theorem}
Fix an integer $n \geq 2$ and let $f: [0,\infty) \to \R$. Then the following are equivalent: 
\begin{enumerate}
\item  We have $f[A] \geq f[B]$ if and only if $A \geq B$ for all positive semidefinite $A, B \in M_n([0,\infty))$.
\item  We have $f[A] > f[B]$ if and only if $A > B$ for all positive definite $A, B \in M_n([0,\infty))$. 
\item There exist $c > 0$ and $d\in \R$ such that $f(x) = cx+d$ for all $x\in [0,\infty)$.
\end{enumerate}
\end{theorem}
\begin{proof}
Clearly $(3) \implies (1)$ and $(3) \implies (2)$.

We begin by showing $(1)\implies (3)$. It suffices to prove the result for $n=2$ as the general case follows by embedding $2 \times 2$ matrices into $M_n([0,\infty))$ via $A \mapsto A \oplus {\bf 0}_{(n-2) \times (n-2)}$. Suppose therefore $(1)$ holds and $n=2$. Set $g(x) := f(x) - f(0)$. Taking $B = {\bf 0}_{2 \times 2}$ shows that $g$ is a positive semidefinite sign preserver. By Theorem \ref{main_thm_1}, there exists $\alpha, \beta > 0$ such that $g(x) = \alpha x^\beta$. Without loss of generality, we can assume $\alpha = 1$. Since $g$ is monotone with respect to the Loewner ordering, by \cite[Theorem 2.4]{fitzgerald1977fractional} we must have $\beta \geq 1$. We claim that the map $A \mapsto f[A]$ is a bijection of the set of $2 \times 2$ positive semidefinite matrices on $M_2([0,\infty))$. Clearly, the map is injective. For a given positive semidefinite matrix $Y \in M_2([0,\infty))$, the matrix $X := Y^{\circ (1/\beta)}$ is also positive semidefinite (Theorem \ref{thm:fitz_horn_fractional}) and $g[X] = Y$. Thus $g[-]$ is surjective. Using the fact that $g[A] \geq g[B]$ if and only if $A \geq B$, it follows that $g^{-1}(x) = x^{1/\beta}$ is also a monotone power function. Thus $\beta \leq 1$ and we conclude $\beta = 1$. 

We now show $(2)\implies (3)$. As above, it suffices to prove this for $n=2$. Set $g(x) := f(x) - f(0)$. For $0 < a < b$, we have $bI_2 > aI_2$ and so $g(b) I_2 > g(a) I_2$. Thus $g$ is increasing on $(0,\infty)$ and $g^+(0) := \lim_{x \to 0^+} g(x)$ exists. For $\epsilon > 0$, observe that $\epsilon I_2 > (\epsilon /4) {\bf 1}_{2 \times 2}$. It follows that $g(\epsilon)I_2 > g(\epsilon/4) {\bf 1}_{2 \times 2}$. Letting $\epsilon \to 0^+$, we obtain $g^+(0) I_2 \geq g^+(0) {\bf 1}_{2 \times 2}$. This is only possible if $g^+(0) = 0$. Thus $g$ is continuous at $x = 0$ and is strictly positive on $(0, \infty)$. 

Let $A \in M_2([0,\infty))$ be positive definite. Then there exists $\epsilon > 0$ such that $A > \epsilon I_2$. Thus $g[A] > g(\epsilon) I_2 > 0$ since $g(\epsilon) > 0$. This proves $g$ preserves positive definiteness. Conversely, suppose $A$ is not positive definite. Then $A \not> \epsilon I_2$ for all $\epsilon > 0$. Thus $g[A] \not> g(\epsilon) I_2$ for all $\epsilon > 0$. Since $g((0,\infty)) = (0,\infty)$, it follows that $g[A]$ is not positive definite. This proves $g$ is a positive definite sign preserver. By Theorem \ref{main_thm_1}, we have $g(x) = \alpha x^\beta$ for some $\alpha, \beta > 0$. Proceeding as in the proof of $(1)\implies (3)$, we conclude that $\beta = 1$.
\end{proof}

\section*{Acknowledgments}
The authors would like to thank Apoorva Khare for his comments on the paper and for suggesting to examine sign preservers for graphs. The authors are grateful to the anonymous referee for a careful reading of the manuscript and for valuable comments and suggestions that improved the paper. Part of this work was completed during a BIRS Research in Teams program at UBC Okanagan. The authors would like to express their gratitude to the Banff International Research Station for its support and to UBC Okanagan for its hospitality. 

D.G. was partially supported by a Simons Foundation collaboration grant for mathematicians and by NSF grant \#2350067. H.G. acknowledges support from PIMS (Pacific Institute for the Mathematical Sciences) Postdoctoral Fellowships. P.K.V. was supported by the Centre de recherches math\'ematiques and Universit\'e Laval (CRM--Laval) Postdoctoral Fellowship, and he acknowledges support from a SwarnaJayanti Fellowship from DST and SERB (Govt.~of India). C.H.Y. was supported in part by an NSERC fellowship. 

\bibliographystyle{plain}
\bibliography{main}

\end{document}